\documentclass[11pt,a4paper]{amsart}

\usepackage[T1]{fontenc}
\usepackage[utf8]{inputenc}
\usepackage{lmodern}

\usepackage[english]{babel}
\usepackage{amsmath,amsfonts,amssymb,amsthm}

\usepackage[pdftex]{graphicx}

\usepackage[bottom=3cm, top=3cm, left=3.5cm, right=3.5cm]{geometry}

\usepackage[foot]{amsaddr}
\usepackage[alphabetic, abbrev]{amsrefs}

\usepackage{mathtools}
\mathtoolsset{showonlyrefs,showmanualtags}	

\usepackage{bm}	
\DeclareMathAlphabet{\mathbbold}{U}{bbold}{m}{n}	

\theoremstyle{plain}
\newtheorem{theorem}{Theorem}
\newtheorem*{theorem*}{Theorem}
\newtheorem{prop}[theorem]{Proposition}
\newtheorem{cor}[theorem]{Corollary}
\newtheorem{lemma}[theorem]{Lemma}
\newtheorem*{lemma*}{Lemma}

\newtheorem{theoremappendix}{Theorem}[section]
\newtheorem{lemmaappendix}[theoremappendix]{Lemma}

\theoremstyle{definition}
\newtheorem{definition}[theorem]{Definition}

\newtheorem{definitionappendix}[theoremappendix]{Definition}

\theoremstyle{remark}
\newtheorem{rmk}[theorem]{Remark}

\newtheorem{rmkappendix}[theoremappendix]{Remark}

\usepackage[bookmarksdepth=3]{hyperref}
\usepackage{xcolor}
\hypersetup{
    colorlinks,
    linkcolor={red!50!black},
    citecolor={blue!50!black},
    urlcolor={blue!80!black}
}

\usepackage{todonotes}

\usepackage{enumitem}
\setlist[itemize]{leftmargin=1cm}

\newcommand{\al}{\alpha}
\newcommand{\mc}[1]{\mathcal{#1}}

\newcommand{\mis}{\mathsf{m}}	

\newcommand{\R}{\mathbb{R}}
\newcommand{\N}{\mathbb{N}}
\newcommand{\distr}{\mathcal{D}}

\newcommand{\eps}{\varepsilon}
\newcommand{\lam}{\lambda}
\newcommand{\J}{\mathcal{J}}
\newcommand{\g}{\gamma}
\newcommand{\wt}{\widetilde}

\newcommand{\cut}{\mathrm{Cut}}

\newcommand{\ver}{\mathcal{V}}

\newcommand{\Vr}{\overline{V}}

\newcommand{\Rr}{R_\mis^N(t)}
\newcommand{\y}{D}		
\renewcommand{\k}{\kappa}
\newcommand{\f}{\mathsf{f}}

\DeclareMathOperator{\tr}{\mathrm{Tr}}

\DeclareMathOperator{\spn}{\mathrm{span}}
\DeclareMathOperator{\rank}{\mathrm{rank}}

\DeclareMathOperator{\diag}{\mathrm{diag}}

\DeclareMathOperator{\Ric}{\mathfrak{Ric}}		
\DeclareMathOperator{\Sec}{\mathrm{Sec}}		

\DeclareMathOperator{\Rcan}{\mathfrak{R}}

\newcommand{\lev}{\alpha}						
\newcommand{\DD}{\mathcal{F}}					
\newcommand{\tanf}{\mathsf{T}}					

\newcommand{\Tor}{\mathrm{Tor}}
\newcommand{\vol}{\mathrm{vol}}
\newcommand{\ganf}{\dot \gamma}

\let\oldtocsection=\tocsection 
\let\oldtocsubsection=\tocsubsection 
\let\oldtocsubsubsection=\tocsubsubsection
 
\renewcommand{\tocsection}[2]{\hspace{0em}\oldtocsection{#1}{#2}}
\renewcommand{\tocsubsection}[2]{\hspace{1em}\oldtocsubsection{#1}{#2}}
\renewcommand{\tocsubsubsection}[2]{\hspace{2em}\oldtocsubsubsection{#1}{#2}}

\author[Davide Barilari]{Davide Barilari$^\flat$}
\address{$^\flat$ Institut de Math\'ematiques de Jussieu-Paris Rive Gauche, UMR CNRS 7586, Universit\'e Paris-Diderot, Batiment Sophie Germain, Case 7012, 75205 Paris Cedex 13, France}
\email{\href{mailto:davide.barilari@imj-prg.fr}{davide.barilari@imj-prg.fr}}

\author[Luca Rizzi]{Luca Rizzi$^\sharp$}
\address{$^\sharp$ Univ. Grenoble Alpes, IF, F-38000 Grenoble, France \newline 
CNRS, IF, F-38000 Grenoble, France}
\email{\href{mailto:luca.rizzi@univ-grenoble-alpes.fr}{luca.rizzi@univ-grenoble-alpes.fr}}

\title[Bakry-\'Emery curvature and model spaces in SR geometry]{Bakry-\'Emery curvature and model spaces in sub-Riemannian geometry}


\date{\today}
\begin{document}

\begin{abstract}
We prove comparison theorems for the sub-Riemannian distortion coefficients appearing in interpolation inequalities. These results, which are equivalent to a sub-Laplacian comparison theorem for the sub-Riemannian distance, are obtained by introducing a suitable notion of sub-Riemannian Bakry-Émery curvature. The model spaces for comparison are variational problems coming from optimal control theory.
As an application we establish the sharp measure contraction property for 3-Sasakian manifolds satisfying a suitable curvature bound.
\end{abstract}

\maketitle

\setcounter{tocdepth}{1}
\tableofcontents

\section{Introduction}

Interpolation inequalities connect different areas of mathematics such as optimal transport, functional inequalities and geometric analysis. Typical examples are the so-called Borell-Brascamp-Lieb inequality, and its geometrical counterpart: the Brunn-Minkowski one. We refer to \cite{Gardner} for a survey of the topic.

A geodesic version of these inequalities has been proved for Riemannian manifolds in the seminal paper \cite{CEMS-interpolations}, provided that the geometry is taken into account through appropriate distortion coefficients.
The main result of \cite{CEMS-interpolations}, written in the form of a Borell-Brascamp-Lieb inequality, reads as follows.
\begin{theorem} \label{t:srbbl-intro}
Let $(M,g)$ be a $n$-dimensional Riemannian manifold, equipped with a smooth measure $\mis$. Fix $t\in [0,1]$. Let $f,g,h:M\to \R$ be non-negative and $A,B\subset M$ be Borel subsets such that $\int_{A}f \,d\mis=\int_{B}g \,d\mis=1$.  Assume that for every $(x,y)\in (A\times B)\setminus \cut(M)$  and $z\in Z_{t}(x,y)$, it holds
\begin{equation}\label{eq:asscems}
\frac{1}{h(z)^{1/n}}\leq \left(\frac{\beta_{1-t}(y,x)}{f(x)}\right)^{1/n}+  \left(\frac{\beta_{t}(x,y)}{g(y)}\right)^{1/n}.
\end{equation}
Then $\int_{M} h\, d\mis \geq 1$.
\end{theorem}

Here $\cut(M)$ denotes the \emph{cut locus} of $(M,g)$, defined as the complement of the subset of $M\times M$ where the squared Riemannian distance is smooth. For $t \in [0,1]$, we denote by $Z_{t}(x,y)$  the set of $t$-intermediate points of geodesics between $x$ and $y$, and by $\beta_t(x,y)$ the \emph{distortion coefficient}
\begin{equation}
\beta_{t}(x,y)=\limsup_{r\to 0}\frac{\mis(Z_{t}(x,\mathcal{B}_{r}(y)))}{\mis(\mathcal{B}_{r}(y))}, \qquad t \in [0,1],
\end{equation}
where $\mathcal{B}_{r}(y)$ denotes the Riemannian ball centred at $y$ of radius $r>0$.  

Distortion coefficients are in general difficult to compute. However, if the Ricci curvature $\mathrm{Ric}_g$ of $M$ is bounded from below, then they can be controlled in terms of the distortion coefficients of suitable Riemannian model spaces.
\begin{theorem} \label{t:comp1}
Let $(M,g)$ be a $n$-dimensional Riemannian manifold, equipped with the Riemannian measure $\mis=\mathrm{vol}_{g}$. Assume that there exists $K\in\R$ such that $\mathrm{Ric}_g(v) \geq (n-1)K$ for every unit vector $v \in TM$. Then for all $t \in [0,1]$ we have
\begin{equation}\label{eq:dist-ref-riem-intro}
\beta_t(x,y) \geq \beta_t^{(K,n)}(x,y), 
\end{equation}
where $\beta_t^{(K,n)}(x,y)$ is the distortion coefficient of the simply connected Riemannian manifold of dimension $n$ and constant sectional curvature $K$.
\end{theorem}
The model coefficients $\beta_t^{(K,n)}(x,y)$ depend only on the distance between $x$ and $y$, and are given explicitly by the following formula:
\begin{equation}\label{eq:coeffespliciti}
\beta_t^{(K,n)}(x,y) = \begin{cases}
t \left(\frac{\sin(t \alpha)}{\sin( \alpha)}\right)^{n-1} & \text{if } K>0, \\ 
t^n & \text{if } K=0 , \\
t\left(\frac{\sinh(t  \alpha)}{\sinh(  \alpha)}\right)^{n-1} & \text{if } K<0,
\end{cases} \qquad \alpha = \sqrt{|K|}d(x,y).
\end{equation}

\medskip
Inequality \eqref{eq:asscems}, with the $\beta$ given by the reference coefficients \eqref{eq:coeffespliciti}, is one of the incarnations of the so-called curvature-dimension condition $\mathrm{CD}(K,N)$, which allows to generalize the concept of Ricci curvature bounded from below (by $K\in \R$) and dimension bounded from above (by $N>1$), to general metric measure spaces.
This is the starting point of the synthetic approach to curvature bounds of Lott-Sturm-Villani \cite{LV-ricci,S-ActaI,S-ActaII} and extensively developed subsequently.

When the Riemannian manifold $(M,g)$ is endowed with an arbitrary smooth measure $\mis=e^{-\psi}\mathrm{vol}_{g}$, where $\psi:M\to \R$ is a smooth function, one should bound, instead, the so-called Bakry-Émery Ricci tensor of parameter $N>n$, defined for every unit vector $v \in TM$ as follows
\begin{equation}\label{eq:beriem}
\mathrm{Ric}_\mis^N(v)=\mathrm{Ric}_{g}(v)+\nabla^{2}\psi(v,v)-\frac{g(\nabla \psi,v)^{2}}{N-n},
\end{equation}
where $\nabla^{2}\psi$ denotes the Riemannian Hessian of $\psi$. The original Bakry-Émery Ricci tensor was introduced for Riemannian manifolds in \cite{Be85} and $N=\infty$ (see also \cite{AGS-BE} for general metric measure spaces). One has then the following result. An equivalent statement can be found in \cite[Appendix A]{WW}.
\begin{theorem} \label{t:comp2}
Let $(M,g)$ be a $n$-dimensional Riemannian manifold, equipped with a smooth volume $\mis$. Assume that there exists $K\in \R$ and $N>n$ such that $\mathrm{Ric}_\mis^N(v)\geq  (N-1)K$ for every unit vector $v\in TM$. Then for all $t \in [0,1]$ we have 
\begin{equation}\label{eq:dist-ref-riem-intro2}
\beta_t(x,y) \geq \beta_t^{(K,N)}(x,y),
\end{equation}
where $\beta_t^{(K,N)}(x,y)$ is defined as in \eqref{eq:coeffespliciti}.
\end{theorem}
The goal of this paper is to extend Theorems \ref{t:comp1} and \ref{t:comp2} to the sub-Riemannian setting. Our analysis suggests that, in this context, model spaces are microlocal, i.e.\ associated to a fixed geodesic, and are not sub-Riemannian manifolds. Rather, they belong to a more general class of variational problems, called linear-quadratic optimal control problems.

The comparison theory for distortion coefficient that we develop here can be paired with the results in \cite{BR-G1}, yielding \emph{explicit} sub-Riemannian Borell-Brascamp-Lieb-type or Brunn-Minkowski-type inequalities, under suitable curvature bounds. This work can be seen as a continuation of \cite{BR-G1}.

\medskip

We now give an overview of our results. We refer to \cite[Sec.~2]{BR-G1} for a minimal introduction to sub-Riemannian geometry, whose conventions are used here. See also Appendix \ref{s:prel}. For comprehensive references see \cite{nostrolibro,riffordbook,montgomerybook}. 

\subsection{Sub-Riemannian geometry and curvature}

A sub-Riemannian manifold is a triple $(M,\distr,g)$ where $g$ is a metric on a smooth vector distribution $\distr$. The sub-Riemannian distance $d_{SR}$ is the infimum of the length of curves tangent to $\distr$. The distribution $\distr$ is required to be bracket generating and under this assumptions $d_{SR}$ is continuous and finite. We assume that $(M,d_{SR})$ is a complete metric space, so that for any pair of points $x,y \in M$ there exists a minimizing geodesic joining them. The interpolating map $Z_t(A,B)$ between two Borel sets $A,B\subset M$ is defined as the $t$-intermediate points of geodesics joining points of $A$ and $B$.

We fix a smooth measure $\mis$ on $M$, and we define the \emph{distortion coefficient} as
\begin{equation}
\beta_{t}(x,y)=\limsup_{r\to 0}\frac{\mis(Z_{t}(x,\mathcal{B}_{r}(y)))}{\mis(\mathcal{B}_{r}(y))}, \qquad t \in [0,1],
\end{equation}
where $\mathcal{B}_{r}(y)$ denotes the sub-Riemannian ball centred at $y$ of radius $r>0$.

We need a directional bracket-generating-type condition, formalized in the following definition (given in terms of a general smooth horizontal curve).
\begin{definition}\label{d:flag}
Let $\gamma:[0,1]\to M$ be a smooth horizontal curve and let $\tanf$ be a smooth horizontal vector field such that $\tanf|_{\gamma(t)} = \dot{\gamma}(t)$ for all $t\in [0,1]$. For $i\geq 1$, let
\begin{equation}
\DD_{\gamma(t)}^i :=  \spn\{(\mathrm{ad}\, \tanf)^j Y|_{\gamma(t)} \mid  Y \in \Gamma(\distr),\; j \leq i-1\} \subseteq T_{\gamma(t)} M, \quad t\in [0,1],
\end{equation}
where $(\mathrm{ad}\, X) Y=[X,Y]$.
The \emph{growth vector} of the curve is the sequence
\begin{equation}\label{eq:growthvect}
\mathcal{G}_{\gamma(t)} := \{\dim \DD_{\gamma(t)}^1,\dim \DD_{\gamma(t)}^2,\ldots\},\qquad t\in [0,T].
\end{equation}
We say that the curve $\gamma$ is:
\begin{itemize}
\item[(a)] \emph{equiregular} if $\dim \DD_{\gamma(t)}^i$ does not depend on $t$ for all $i \geq 1$,
\item[(b)] \emph{ample} if for all $t$ there exists $m \geq 1$ such that $\dim \DD_{\gamma(t)}^{m} = \dim T_{\gamma(t)}M$.
\end{itemize}
\end{definition}
If $\gamma$ is ample and equiregular, then the following objects are well defined (we refer to Appendix~\ref{s:prel} for details):
\begin{itemize}
\item a Young diagram $\y$, encoding the growth vector of $\gamma$;
\item a quadratic form $\mathfrak{R}_{\gamma}(t):T_{\gamma(t)}M \times T_{\gamma(t)}M \to \R$ defined along $\gamma$;
\item a scalar product $\langle \cdot |\cdot\rangle_{\gamma(t)}$, on $T_{\gamma(t)}M$, extending $g$ along $\gamma$;
\item a canonical moving frame $X_1(t),\dots,X_n(t)$ along $\gamma$, orthonormal with respect to $\langle\cdot|\cdot\rangle_{\gamma(t)}$, and adapted to the flag $\DD_{\gamma(t)}$.
\end{itemize}
The canonical moving frame is a generalization of the concept of parallel transported frame. It is uniquely defined up to constant orthogonal transformations respecting the structure of the flag $\DD^i_{\gamma(t)}$. It is obtained as the projection of the so called canonical frame introduced in \cite{lizel}, in the setting of Jacobi curves \cite{agzel1,agzel2}.

\begin{rmk}\label{r:riemann}
Every Riemannian geodesic is ample and equiregular. In this case $\mathfrak{R}_{\gamma}(t)(v,v)=R_{g}(v,\dot \gamma(t),\dot\gamma(t),v)$, where $R_{g}$ is the Riemann curvature tensor. Furthermore, in the Riemannian case, $\mathfrak{R}_{\gamma}(t)$ is quadratic also with respect to $\dot{\gamma}(t)$.  Notice that $\tr \mathfrak{R}_{\gamma}(t)= \mathrm{Ric}_{g}(\dot \gamma(t))$.
\end{rmk}

\begin{definition}\label{def:geodvolder}
Given a smooth measure $\mis$ and an ample and equiregular geodesic $\g$, we define the \emph{geodesic volume derivative} along $\g$ as the function
\begin{equation}
\rho_{\mis,\g}(t)= \frac{d}{dt}\log \mis_{\gamma(t)}(X_1(t),\dots,X_n(t)),
\end{equation}
where $X_{1},\ldots,X_{n}$ is a canonical moving frame along $\gamma$.
\end{definition}
\begin{rmk}
In the Riemannian case $X_1(t),\dots,X_n(t)$ is parallel and orthonormal, hence if $\mis = e^{-\psi}\vol_g$, then $\rho_{\mis,\g}(t) = -g(\dot\gamma(t),\nabla\psi)$. In particular $\rho_{\vol_g,\g}(t)=0$ along any geodesic. For a definition of the geodesic volume derivative not using frames, and its relation with curvature we refer to \cite{ABP}. 
\end{rmk}
The  scalar product $\langle \cdot |\cdot\rangle_{\gamma(t)}$  induces a quadratic form $\mathfrak{B}_{\gamma}(t):T_{\gamma(t)}M \times T_{\gamma(t)}M\to \R$, where $\mathfrak{B}_{\gamma}(t)(v,w)$ is the scalar product of the orthogonal projections of $v,w$ on $\distr_{\gamma(t)}$.

\begin{definition}
Let $(M,\distr,g)$ be a $n$-dimensional sub-Rieman\-nian manifold. The \emph{Bakry-Émery curvature} along $\gamma$ is the family of quadratic forms $\mathfrak{R}_{\mis,\gamma}^{N}(t):T_{\gamma(t)}M \times T_{\gamma(t)}M\to \R$ defined by
\begin{equation} \label{eq:bemery}
\mathfrak{R}_{\mis,\gamma}^{N}(t)= 
\mathfrak{R}_{\gamma}(t)-\left(\frac{\dot \rho_{\mis,\g}(t)}{k}+\frac{n}{N-n}\frac{ \rho_{\mis,\g}^{2}(t)}{k^{2}}\right)\mathfrak{B}_{\gamma}(t),
\end{equation} 
where $N>n$ is a real parameter, and $k=\rank \distr$.
\end{definition}
\begin{rmk}
In the Riemannian case $k=n$ and  $\mathfrak{B}_{\gamma}(t)$ coincides with the Riemannian metric on $T_{\gamma(t)}M$. Hence:
\begin{equation}\label{eq:bemeryricci}
\tr \mathfrak{R}_{\mis,\gamma}^{N}(t)= 
\tr \mathfrak{R}_{\gamma}(t)-\dot \rho_{\mis,\g}(t)-\frac{1}{N-n} \rho_{\mis,\g}^{2}(t).
\end{equation}
Letting $\mis=e^{-\psi}\mathrm{vol}_{g}$, we have $\rho_{\mis,\g}(t)=-g(\nabla\psi,\dot \gamma(t))$ and therefore $\dot \rho_{\mis,\g}(t)=-\nabla^{2}\psi(\dot \gamma(t),\dot \gamma(t))$. Hence \eqref{eq:bemeryricci} reduces to the classical Bakry-Émery Ricci curvature defined in \eqref{eq:beriem}.
\end{rmk}

\subsection{The model spaces} 

We now introduce model spaces, which are associated to a fixed geodesic. Let $A,B$ be $n\times n$ matrices, with $B \geq 0$ and symmetric. Their special form is determined by the Young diagram $\y$ of the geodesic. Letting $k\leq n$ be the rank of $B$, there exist vectors $b_1,\dots,b_k \in \R^n$, unique up to orthogonal transformation, such that $B = \sum_{i=1}^k b_i b_i^*$.

Let $Q$ be a symmetric $n\times n$ matrix (playing the role of curvature bound). We consider a variational problem on $\R^n$, that consists in minimizing the functional
\begin{equation}\label{eq:lq1i}
C(u) = \frac{1}{2}\int_{0}^{1} \left(u^* u - x^*Qx \right)dt,
\end{equation}
among all trajectories $x:[0,1] \to \R^n$ with fixed endpoint satisfying
\begin{equation}\label{eq:lq2i}
\dot{x} = Ax +\sum_{i=1}^k u_i b_i,
\end{equation} 
for some control $u\in L^2([0,1],\R^k)$. These models are called \emph{linear quadratic optimal control problems} in control theory (see Section \ref{s:LQ} for details).

The functional \eqref{eq:lq1i} does not define a metric spaces structure on $\R^n$, in general. However one can still define the set $Z_t(\Omega_0,\Omega_1)$ of $t$-intermediate points between two Borel sets $\Omega_0,\Omega_1 \subset \R^{n}$ as the set of all points $x(t)$, where $x:[0,1]\to \R^n$ is a minimizer for the problem \eqref{eq:lq1i}-\eqref{eq:lq2i} such that $x(0) \in \Omega_0$ and $x(1)\in \Omega_1$. Then we define the model distortion coefficient as
\begin{equation}\label{eq:dclq}
\beta_{t}^{\y,Q}=\limsup_{r\to 0}\frac{|Z_{t}(x,B_{r}(y))|}{|B_{r}(y)|}, \qquad t \in [0,1],
\end{equation}
where $x,y\in \R^n$, $B_r(y)$ denotes the Euclidean ball with center $y$ and radius $r>0$, and $|\cdot|$ denotes the Lebesgue measure of $\R^n$. The quantity in the right hand side of \eqref{eq:dclq} is independent on the choice of $x,y\in \R^n$, and so the definition is well posed.

 The distortion coefficients of a LQ model can be easily computed by solving a linear Hamiltonian system, once the matrices $A,B,Q$ are fixed (cf. Proposition \ref{p:LQdist}).

\begin{rmk}
If $\y$ is the Young diagram of a geodesic on a $n$-dimensional Riemannian manifold, then $A=\mathbbold{0}_n$, $B=\mathbbold{1}_n$.  If we choose $Q= \kappa \mathbbold{1}_n$, we obtain the homogeneous distortion coefficient (cf.\ Section~\ref{ex:1})
\begin{equation}
\beta_t^{\y,Q} = \begin{cases}
\left(\frac{\sin(t\alpha)}{\sin(\alpha)}\right)^n & \k>0,  \\
t^n & \k=0, \\
\left(\frac{\sinh(t\alpha)}{\sinh(\alpha)}\right)^n & \k<0,
\end{cases} \qquad \alpha = \sqrt{|\k|}.
\end{equation}
One can recover the sharp Riemannian model coefficient $\beta_t^{(K,n)}$ of \eqref{eq:coeffespliciti} by choosing, instead, the $n\times n$ matrix $Q = K d^2(x,y) \diag(1,\dots,1,0)$. Since the potential $Q$ mimics the effect of curvature, this choice correctly takes into account that there is no curvature in the direction of the motion.
\end{rmk}

\subsection{Sectional-type comparison results} 

We now state the first pair of main results of the paper. Theorem \ref{t:srbe} requires separate assumptions on the curvature and on the volume derivative. Theorem \ref{t:srbe2} unifies both assumptions in a single Bakry-Émery-type lower bound.

\begin{theorem} \label{t:srbe}
Let $(x,y)\notin \cut(M)$ and assume that the unique geodesic $\gamma$ joining $x$ and $y$ is ample and equiregular, with Young diagram $\y$. Assume that the geodesic volume derivative satisfies $\rho_{\mis,\g}(t)\leq 0$ along $\gamma$, and that there exists a symmetric $n\times n$ matrix $Q$ such that $\mathfrak{R}_{\gamma}(t)\geq Q$ for every $t\in [0,1]$. Then  
\begin{equation}\label{eq:srbeeq1}
\frac{\beta_t(x,y)}{\beta_t^{\y,Q}} \text{ is a non-increasing function of $t \in (0,1]$}.
\end{equation}
In particular we have
\begin{equation}\label{eq:srbeeq2}
\beta_{t}(x,y)\geq \beta_{t}^{\y,Q}, \qquad \forall t\in [0,1].
\end{equation}
If, instead, $\mathfrak{R}_{\gamma}(t)\leq Q$ and $\rho_{\mis,\g}(t)\geq 0$ along $\gamma$ for every $t\in [0,1]$, then the function in \eqref{eq:srbeeq1} is non-decreasing and \eqref{eq:srbeeq2} holds with the opposite inequality.
\end{theorem}
The inequality $\Rcan_{\gamma}(t) \geq Q$ is understood by identifying the quadratic form $\Rcan_{\gamma}(t)$ with a $n \times n$ matrix using a canonical frame $X_1(t),\dots,X_n(t)$. 

\begin{rmk}\label{r:cgeqleq}
In  Theorem \ref{t:srbe} the assumption $\rho_{\mis,\g}(t) \leq 0$ (resp.\ $\geq 0$) can be weakened to $\rho_{\mis,\g}(t) \leq c$ for some $c\in\R$ with the following modifications in the conclusion:
\begin{equation}
\frac{\beta_t(x,y)}{\beta_t^{\y,Q}}e^{- ct} \text{ is a non-increasing function of $t\in(0,1]$}.
\end{equation}
In particular we have
\begin{equation}
\beta_t(x,y) \geq \beta_t^{\y,Q} e^{c(t-1)}, \qquad \forall t\in[0,1],
\end{equation}
and similarly with reversed inequalities  if $\mathfrak{R}_{\gamma}(t) \leq Q$ and $\rho_{\mis,\g}(t) \geq c$ for $t\in [0,1]$.
\end{rmk}

\begin{theorem} \label{t:srbe2}
Let $(x,y)\notin \cut(M)$ and assume that the unique geodesic $\gamma$ joining $x$ and $y$ is ample and equiregular, with Young diagram $\y$.  Assume that there exists $N>n$ and a symmetric $n\times n$ matrix $Q$ such that $\frac{1}{N}\mathfrak{R}_{\mis,\gamma}^{N}(t)\geq \frac{1}{n}Q$ for every $t\in [0,1]$. Then
\begin{equation}
\frac{\beta_{t}(x,y)^{1/N}}{(\beta_{t}^{D,Q})^{1/n}}\text{ is a non-increasing function of $t \in (0,1]$}.
\end{equation}
In particular we have
\begin{equation} \label{eq:dimint2}
\beta_{t}(x,y)^{1/N}\geq (\beta_{t}^{\y,Q})^{1/n}, \qquad \forall t\in [0,1].
\end{equation}
\end{theorem}
Notice that \eqref{eq:dimint2} gives a dimensional interpretation of the parameter $N$. Indeed, the distortion coefficient can be compared with the model one only after they are both normalized by an effective dimension.

\subsection{Ricci-type comparison results}

The sectional-type curvature bounds in the assumptions of Theorems \ref{t:srbe} and \ref{t:srbe2} can be weakened to Ricci-type bounds, similar in spirit to the ones in Theorems~\ref{t:comp1} and \ref{t:comp2}. In the Riemannian case, this is done by taking the trace of the matrix Riccati equation describing the evolution of Jacobi fields, and turning it into a simple scalar Riccati inequality (see e.g.\ \cite[Ch.\ 14]{V-oldandnew}). In the sub-Riemannian case, the process of ``taking the trace'' is more delicate. Due to the anisotropy of the structure, it only makes sense to take partial traces, leading to a number of Ricci curvatures (each one obtained as a partial trace on an invariant subspace of $T_{\gamma(t)} M$, determined by the Young diagram $\y$). This is done  by using a tracing technique developed in \cite{BR-connection}.

\begin{figure}
\centering
\includegraphics[scale=1]{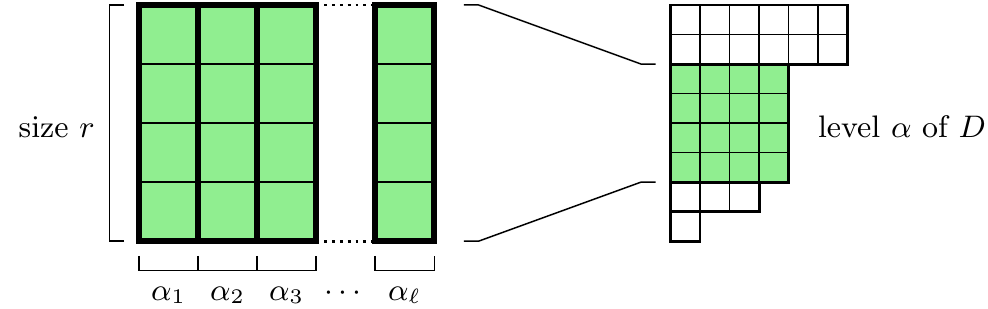}
\caption{Level $\lev$ and superboxes $\lev_i$ of a Young diagram.}
\label{fig:levyd}
\end{figure}

In order to state our main results, we need to introduce some terminology related to the boxes of a Young diagram $\y$ associated with an ample and equiregular geodesic. We refer to Figure~\ref{fig:levyd}. A \emph{level} is the collection of all the rows of the Young diagram with the same length. A \emph{superbox} is the collection of all boxes of the Young diagram in a given level, belonging to the same column. The \emph{size} of a level or a superbox is the number $r$ of boxes in each of its columns.  For a given level $\lev$ of the Young diagram, of length $\ell$, we denote its superboxes as $\lev_{1},\ldots,\lev_{\ell}$. Every superbox $\lev_i$ is associated with an invariant subspace $S^{\lev_{i}}_{\gamma(t)} \subseteq T_{\gamma(t)}M$, of dimension equal to its size. Finally, for each superbox $\lev_i$, we define a sub-Riemannian Ricci curvature (resp.\ Bakry-Émery Ricci) denoted $\Ric^{\lev_i}_{\gamma}(t)$ (resp.\  $\Ric_{\mis,\gamma}^{N, \lev_{i}}(t)$) for $i=1,\ldots,\ell$,
\begin{equation}\label{eq:srbakryemery}
\Ric_{\gamma}^{\lev_{i}}(t)=\tr \left(\mathfrak{R}_{\gamma}(t)\big|_{S^{\alpha_{i}}_{\gamma(t)}}\right),\qquad \Ric_{\mis,\gamma}^{N, \lev_{i}}(t)=\tr \left(\mathfrak{R}_{\mis,\gamma}^{N}(t)\big|_{S^{\alpha_{i}}_{\gamma(t)}}\right).
\end{equation}
Thus, we have a total number of Ricci curvatures equal to the number of superboxes of the Young diagram. In the Riemannian case, the Young diagram has a single column with $n = \dim M$ boxes. Thus there is only one superbox, and one Ricci curvature, corresponding to the full trace of $\Rcan_{\gamma}(t)$. See Section~\ref{s:cancurv} for details.

In the following theorems, $\Upsilon$ denotes the set of levels of the Young diagram.

\begin{theorem}\label{t:srbericci1}
Let $(x,y)\notin \cut(M)$ and assume that the unique geodesic joining $x$ and $y$ is ample and equiregular, with Young diagram $\y$. 

Assume that $\rho_{\mis,\g}(t)\leq 0$ along $\gamma$ and that for every level $\lev$ of size $r_{\lev}$ and length $\ell_{\lev}$ of $\y$ there exist $\k_{\lev_{i}}\in \R$, for $i=1,\ldots,\ell_{\lev}$, such that for every superbox $\lev_{i}$
\begin{equation}
\frac{1}{r_{\lev}}\Ric_{\gamma}^{\lev_{i}}(t) \geq \k_{\lev_{i}},  \qquad \forall t\in [0,1].
\end{equation}
Then 
\begin{equation}
\frac{\beta_t(x,y)}{\displaystyle{\prod_{\lev \in \Upsilon}} \left(\beta_{t}^{\y_{\lev},Q_{\lev}}\right)^{r_{\lev}}}\text{ is a non-increasing function of $t \in (0,1]$},
\end{equation}
where $\y_{\lev}$ is the Young diagram composed by a single row of length $\ell$, and $Q_{\lev}=\mathrm{diag}(\k_{\lev_{1}},\ldots,\k_{\lev_{\ell_{\alpha}}})$. In particular
\begin{equation}
\beta_{t}(x,y)\geq \prod_{\lev \in \Upsilon} \left(\beta_{t}^{\y_{\lev},Q_{\lev}}\right)^{r_{\lev}},  \qquad \forall t\in [0,1].
\end{equation}
\end{theorem}

A similar conclusion can be obtained if, in Theorem~\ref{t:srbericci1}, one assumes $\rho_{\mis,\g}(t)\leq c$ for some $c\in \R$ along $\gamma$ (cf.\ Remark~\ref{r:cgeqleq}).

\begin{theorem}\label{t:srbericci2}
Let $(x,y)\notin \cut(M)$ and assume that the unique geodesic joining $x$ and $y$ is ample and equiregular, with Young diagram $\y$. 

Assume that there exists $N>n$ such that for every level $\lev$ of size $r_{\lev}$ and length $\ell_{\lev}$ of $\y$ there exist $\k_{\lev_{i}}\in \R$, for $i=1,\ldots,\ell_{\lev}$, such that for every superbox $\lev_{i}$  
\begin{equation}
\frac{1}{r_{\lev}}\Ric_{\mis,\gamma}^{N, \lev_{i}}(t) \geq \frac{N}{n}\k_{\lev_{i}},   \qquad \forall t\in [0,1].
\end{equation}
Then  
\begin{equation}
\frac{\beta_t(x,y)^{1/N}}{\displaystyle\prod_{\lev \in \Upsilon} \left(\beta_{t}^{\y_{\lev},Q_{\lev}}\right)^{r_{\lev}/n}}\text{ is a non-increasing function of $t \in (0,1]$},
\end{equation}
where $\y_{\lev}$ is the Young diagram composed by a single row of length $\ell$, and $Q_{\lev}=\mathrm{diag}(\k_{\lev_{1}},\ldots,\k_{\lev_{\ell_{\alpha}}})$. In particular
\begin{equation}
\beta_{t}(x,y)^{1/N}\geq  \prod_{\lev \in \Upsilon} \left(\beta_{t}^{\y_{\lev},Q_{\lev}}\right)^{r_{\lev}/n},  \qquad \forall t\in [0,1].
\end{equation}
\end{theorem}

\subsection{Removing the direction of motion}

Theorems \ref{t:srbericci1} and \ref{t:srbericci2} do not take into account the fact that distances are not distorted in the direction of a geodesic. This is well known in Riemannian geometry (see e.g.\ the discussion in \cite[p.\ 384]{V-oldandnew}). It corresponds to the fact that $\mathfrak{R}_{\gamma}(t)(\dot\gamma(t),\dot\gamma(t)) = 0$, which remains true in sub-Riemannian geometry as a consequence of the homogeneity of the Hamiltonian. 

At a technical level, the distortion coefficient can always be written as
\begin{equation}
\beta_t(x,y) = t \beta_t^\perp(x,y), \qquad \forall\,(x,y)\notin \cut(M),
\end{equation}
where $\beta_t^\perp(x,y)$ is, roughly speaking, the distortion felt in the transverse directions to the geodesic joining $x$ with $y$. In all proofs, the direction of the motion can be factored out, proving comparison results for $\beta_t^\perp(x,y)$. In terms of Young diagram, the direction of the motion corresponds to a block situated in the bottom level, the only one of length $1$, whose effective size is reduced by one. We omit the details, recording only the final statement, which is a sharper version of  Theorem \ref{t:srbericci2}.

\begin{theorem}\label{t:srbericci-sharp}
Let $(x,y)\notin \cut(M)$ and assume that the unique geodesic joining $x$ and $y$ is ample and equiregular, with Young diagram $\y$. 

Assume that there exists $N>n$ such that for every level $\lev$ of size $r_{\lev}$ and length $\ell_{\lev}$ there exist $\k_{\lev_{i}}\in \R$, for $i=1,\ldots,\ell_{\lev}$, such that for every superbox $\lev_{i}$  
\begin{equation}
\frac{1}{r_{\lev}}\Ric_{\mis,\gamma}^{N, \lev_{i}}(t) \geq \frac{N-1}{n-1}\k_{\lev_{i}},   \qquad \forall t\in [0,1],
\end{equation}
with the convention that if $\lev$ is the level of length $1$ then $r_\lev$ is replaced by $r_\lev-1$, and if $r_\lev =0$ then this level is omitted. Then, with the same convention, we have
\begin{equation}
\frac{\beta_t^\perp(x,y)^{1/(N-1)}}{\displaystyle\prod_{\lev \in \Upsilon} \left(\beta_{t}^{\y_{\lev},Q_{\lev}}\right)^{r_{\lev}/(n-1)}}\text{ is a non-increasing function of $t \in (0,1]$},
\end{equation}
where $\y_{\lev}$ is the Young diagram composed by a single row of length $\ell$, and $Q_{\lev}=\mathrm{diag}(\k_{\lev_{1}},\ldots,\k_{\lev_{\ell_{\alpha}}})$. In particular
\begin{equation}
\beta_{t}(x,y)^{1/(N-1)}\geq  t^{1/(N-1)}\prod_{\lev \in \Upsilon} \left(\beta_{t}^{\y_{\lev},Q_{\lev}}\right)^{r_{\lev}/(n-1)},  \qquad \forall t\in [0,1].
\end{equation}
\end{theorem}
\begin{rmk}\label{rmk:formally}
If $\rho_{\mis,\g} \leq 0$ along the geodesic joining $x$ with $y$, then one can take formally $N=n$ in the previous theorem, and obtain a version of Theorem \ref{t:srbericci1} with the direction of the motion taken out. For an $n$-dimensional Riemannian manifold, Theorem \ref{t:srbericci-sharp} recovers the sharp statements of Theorems \ref{t:comp1} and \ref{t:comp2}.
\end{rmk}

\subsection{The two columns case}

As a consequence of Theorem \ref{t:srbericci-sharp}, and non-trivial inequalities for the model distortion coefficients, we obtain polynomial bounds for the distortion coefficient under appropriate curvature bounds when the Young diagram has two columns. We only give a statement for $\rho_{\mis,\g} \leq 0$, in which case the Bakry-Émery curvature is not necessary (formally $N=n$ in Theorem \ref{t:srbericci-sharp}). We adopt an ad-hoc labelling notation for the superboxes of a $2$-columns Young diagram and the corresponding Ricci curvatures, as in Figure \ref{f:YdMCP}.

\begin{figure}[h]
\centering
\includegraphics[scale=1]{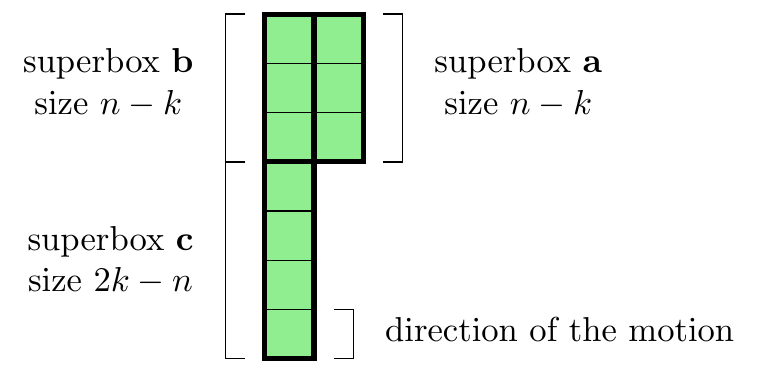}
\caption{Young diagram with two columns. Here, $k$ is the rank of the sub-Riemannian distribution, while $n$ is the dimension of the manifold. The first level, of size $n-k$ and length $2$, is composed by the superboxes denoted, respectively, $b$ and $a$. The second level, of size $2k-n$ and length $1$, is composed by the superbox $c$.}\label{f:YdMCP}
\end{figure}

\begin{theorem}\label{t:mcp}
Let $(x,y)\notin \cut(M)$ and assume that the unique geodesic joining $x$ and $y$ is ample and equiregular, with Young diagram $\y$ as in Figure \ref{f:YdMCP}. Assume that for all $t\in[0,1]$ we have $\rho_{\mis,\g}(\gamma(t))  \leq 0$ and
\begin{align}
\Ric^a_{\gamma(t)} & \geq (n-k) \kappa_a,  \\
\Ric^b_{\gamma(t)} & \geq (n-k) \kappa_b,  \\
\Ric^c_{\gamma(t)} & \geq  (2k-n-1)\kappa_c, 
\end{align}
for some $\k_a,\k_b,\k_c \in \R$ satisfying	
\begin{equation}
4\kappa_a+\kappa_b^2 \geq 0, \qquad \kappa_b \geq 0, \qquad \kappa_c \geq 0. \label{eq:conditions2}
\end{equation}
Then $\beta_t(x,y)/t^{k+3(n-k)}$ is a non-increasing of $t\in (0,1]$, and hence
\begin{equation}\label{eq:mcptheorem}
\beta_t(x,y) \geq t^{k+3(n-k)}.
\end{equation}
The exponent $k+3(n-k)$ is optimal, i.e.\ the lowest one such that \eqref{eq:mcptheorem} holds true.
\end{theorem}
\begin{rmk}
For fat distributions\footnote{A distribution $\distr$ is fat if for any non-zero $X\in \Gamma(\distr)$, $TM$ is locally generated by $X$ and $[X,\distr]$.} of rank $k$ on a $n$-dimensional manifold, all non-trivial geodesics have the same Young diagram with two columns. The exponent $\mathcal{N} = k+3(n-k)$ is equal to the \emph{geodesic dimension} of the sub-Riemannian manifold, defined in \cite{curvature} (see also \cite{R-MCP} for a definition on metric measure spaces).
\end{rmk}

We apply our results to Sasakian and $3$-Sasakian structures in Section \ref{s:applications}, to which we refer for precise statements. For brevity we present here, as an example, the main result concerning 3-Sasakian structures.

\begin{theorem}\label{t:3Sas}
Let $(M,\distr,g)$ be a 3-Sasakian manifold of dimension $4d+3$, equip\-ped with its canonical measure.
Assume that, for every non-zero $X \in \distr$
\begin{equation}\label{eq:boundrhoa}
\Sec(X \wedge Y) \geq K \geq -9, \qquad \forall\, Y \in \spn\{\phi_I X,\phi_J X,\phi_K X\},
\end{equation}
where $\Sec$ is the Riemannian sectional curvature of the $3$-Sasakian structure. Then $\beta_t(x,y)/t^{4d+9}$ is a non-increasing of $t\in (0,1]$ and $(x,y)\notin \cut(M)$. In particular
\begin{equation}
\beta_t(x,y) \geq t^{4d+9}, \qquad \forall t\in [0,1],
\end{equation}
and the exponent is optimal.
\end{theorem}

As a consequence of \cite[Thm.\ 9]{BR-G1}, the bounds $\beta_t(x,y) \geq t^N$ in the above statements are equivalent to a weighted Brunn-Minkowski inequality of the form
\begin{equation}
\mis(Z_t(A,B))^{1/n} \geq (1-t)^{N/n} \mis(A)^{1/n} + t^{N/n}\mis(B)^{1/n}, \qquad \forall t\in[0,1],
\end{equation}
for all Borel sets $A,B \subset M$, and to the $\mathrm{MCP}(0,N)$. They are also equivalent to suitable (sub-)Laplacian comparison theorems, see Section \ref{s:laplacian}. 

In particular, Theorem \ref{t:3Sas} has the consequence that all 3-Sasakian manifolds of dimension $4d+3$ satisfying suitable curvature bounds satisfy the $\mathrm{MCP}(0,N)$ for all $N\geq 4d+9$. This results contributes to the list of sub-Riemannian structures satisfying a measure contraction properties \cite{Riff-Carnot,BR-MCP,R-MCP,BR-MCPHtype,AAPL-3D,LLZ-Sasakian}.

\subsection{Equivalence to sub-Laplacian comparison}\label{s:laplacian}

Comparison results for distortion coefficients are equivalent to comparison theorem for the sub-Laplacian of the sub-Riemannian distance. This fact is known, and its proof in specific cases can be found in \cite[Prop.\ 9]{BR-MCP} and in \cite[Sec.\ 6.2]{Ohta14} for the Riemannian case.

In what follows $\Delta_{\mis}$ denotes the sub-Laplacian on $M$ associated with the measure  $\mis$, i.e.\ the generator of the Dirichlet form $Q(u)=\int_{M}\|\nabla_{SR} u\|^{2}d\mis$, where $\nabla_{SR} u$ is the sub-Riemannian gradient of $u$.
\begin{theorem}
Let $(M,\distr,g)$ be an ideal\footnote{A sub-Riemannian manifold is ideal if it does not admit non-trivial singular minimizing geodesics. This hypothesis can be omitted if, instead, we assume that $\gamma(t) \notin \cut(x)$ for all $t\in (0,1]$.} sub-Riemannian manifold. Let $y \notin \cut(x)$, and let $\gamma:[0,1]\to M$ be the unique geodesic joining $x$ with $y$. Then, letting $\f_x(\cdot) = \tfrac{1}{2}d^2_{SR}(x,\cdot)$, we have
\begin{equation}
\Delta_{\mis} \f_x(\gamma(t)) = t\frac{d }{dt} \log\beta_t(x,y),  \qquad \forall t\in(0,1].
\end{equation}
Hence, for any smooth $h :(0,1] \to \R_+$, with $h(1) =1$, the following are equivalent:
\begin{itemize}
\item the function $t \mapsto \beta_t(x,y)/h(t)$ is non-increasing on $(0,1]$;
\item $\Delta_{\mis} \f_x(\gamma(t)) \leq t \frac{d}{dt}\log h(t)$ for all $t \in (0,1]$.
\end{itemize}
In particular, both statements imply that $\beta_t(x,y) \geq h(t)$ for all $t \in [0,1]$.
\end{theorem}
\begin{proof}
The first formula is a consequence of the fact that $\dot\gamma(t) = \nabla_{SR} \f_x(\gamma(t))$, and of the definition of divergence. See for example \cite[Prop.\ B.1]{BR-Duke}, or also the proof in  \cite[Prop.\ 9]{BR-MCP}. The remaining implications are obvious.
\end{proof}

\subsection{Maximal length bounds}

It is worth mentioning that, from the proof of the above theorems, one can recover a bound for the maximal length of minimizing geodesics, obtained in \cite{BR-comparison}, to which we refer for details.

\begin{theorem}
Let $\gamma:[0,T] \to M$ be a length-parametrized minimizing geodesic, ample and equiregular, with Young diagram $\y$. Assume that there exists a level $\lev$ with size $r_\lev$ and length $\ell_\lev$, and $\k_{i}\in \R$, for $i=1,\ldots,\ell_{\lev}$, such that
\begin{equation}
\frac{1}{r_\lev}\Ric_{\gamma}^{\lev_i}(t) \geq \k_{i},   \qquad \forall t\in [0,T],
\end{equation}
with the convention that, if $\lev$ is the level of length $1$, then $r_\lev$ is replaced by $r_\lev-1$. Then $\ell(\g) \leq t_c(\k_{1},\ldots,\k_{\ell_\lev})$, where the latter is the first conjugate time of the LQ problem whose Young diagram has a single row of length $\ell$ and $Q=\mathrm{diag}(\k_{1},\ldots,\k_{\ell_\lev})$.
\end{theorem}

Conditions on $\k_1,\ldots,\k_{\ell_\lev}$ such that $t_c(\k_{1},\ldots,\k_{\ell_\lev}) < +\infty$ can be found simply applying the main result in \cite{ARS-LQ}. We just give two examples. If $\lev$ is a level of length $\ell_\alpha =1$, then the condition is $\kappa_1 >0$, in which case 
\begin{equation}
t_c(\k_1) = \frac{\pi}{\sqrt{\k_1}}.
\end{equation}
If $\lev$ is a level of length $\ell_\alpha=2$, the conditions are
\begin{equation}
\begin{cases} \k_1 >0, \\
\k_1^2 +4\k_2 >0,
\end{cases} \qquad \text{or} \qquad \begin{cases} \k_1 \leq 0, \\
\k_2 >0,
\end{cases}
\end{equation}
in which case
\begin{equation}
t_c(\kappa_1,\kappa_2) \leq \frac{2\pi}{\mathrm{Re}(\sqrt{x+y}-\sqrt{x-y})},\qquad x=\frac{\k_1}{2}, \quad y = \frac{\sqrt{\k_1^2+4\k_2}}{2}.
\end{equation}
These results yield the sharp diameter of the standard sub-Riemannian structure on the Hopf fibrations $\mathbb{S}^1 \hookrightarrow \mathbb{S}^{2d+1} \to \mathbb{CP}^d$ and the quaternionic Hopf fibrations $\mathbb{S}^3 \hookrightarrow \mathbb{S}^{2d+3} \to \mathbb{HP}^d$.

\subsection{Related literature}

The notion of sub-Riemannian curvature we consider in this paper has been developed starting from the pioneering works of Agrachev-Zelenko \cite{agzel1,agzel2} and Zelenko-Li \cite{lizel} and then subsequently developed in \cite{curvature,BR-comparison,BR-connection}. Several applications of this theory, such as Bonnet-Myers theorems and measure contraction properties, have been given in the recent years in \cite{AAPL-3D,LLZ-Sasakian,BR-comparison,RS17,ABR-contact,BI19}.

A different approach to sub-Riemannian curvature, based on the extension of Bochner-type formulas and curvature-dimension inequalities, has been proposed by Baudoin, Garofalo and collaborators (see \cite{BG-CD}, \cite{Baudoin-EMS} and references therein). This technique has been implemented efficiently for sub-Riemannian structures induced on the horizontal bundle of totally geodesic Riemannian foliations.

An analysis based on the canonical variation of the index form has been successfully implemented on Sasakian foliations in \cite{BGKT-Sasakian}. The same idea is applied in \cite{BGMR-Htype-comparison} to H-type foliations with parallel Clifford structure, introduced in \cite{BGMR-Htype-structure}, which extends to higher corank all the nice features of Sasakian foliations. Let us mention also \cite{Rumin,Chanillo,Hughen} for a related approach, based on the Riemannian Jacobi equation, on $3$-dimensional contact manifolds.

The study of interpolation inequalities in the Heisenberg group have been initiated in \cite{BKS-geomheis} (see also \cite{BKS-geomcorank1} for corank $1$ Carnot groups). Then, in \cite{BR-G1}, the authors proved that any ideal sub-Riemannian manifold supports interpolation inequalities, provided that the geometry is taken into account through suitable distortion coefficients. These works were motivated by the first crucial observation that classical Brunn-Minkowski type inequalities modelled on Riemannian space forms are not satisfied in the sub-Riemannian setting \cite{Juillet} (see also the recent \cite{JuilletBM}). 

For simplicity in this paper we focus only on the sub-Riemannian case. Nevertheless, the concept of curvature used here is purely Hamiltonian and permits to recover analogue results in the Finsler setting, such as those considered in \cite{Ohta09,Wu07}.

\subsection{Structure of the paper}
Models spaces are explained in Section \ref{s:LQ}, while the model distortion coefficient and its properties are studied for some special cases in Section \ref{s:somemodels}. The main results are proved in Sections \ref{s:proofs} and \ref{s:proofsricci}. There we use two important technical ingredients: the theory of sub-Riemannian curvature  and canonical moving frames, and general comparison theory for matrix Riccati equations with limit initial data. Readers who are not familiar with these tools are advised to go through Appendices \ref{s:prel} and \ref{s:riccati}, respectively, before reading Section \ref{s:proofs} and \ref{s:proofsricci}. Finally, in Section \ref{s:applications}, we specify our results to Sasakian and $3$-Sasakian structures. 

\subsection*{Acknowledgements}
This work was supported by the Grants ANR-15-CE40-0018 and ANR-18-CE40-0012 of the ANR, and by a public grant as part of the Investissement d'avenir project, reference ANR-11-LABX-0056-LMH, LabEx LMH, in a joint call with the ``FMJH Program Gaspard Monge in optimization and operation research''. 

\section{Linear Quadratic problems}\label{s:LQ}

Linear quadratic optimal control problems (LQ in the following) are a classical topic in optimal control theory. They are variational problems in $\R^n$ with a quadratic cost  and linear dynamics. We briefly recall their general features, and we refer to \cite[Ch.\ 16]{Agrachevbook}, \cite[Ch.\ 1]{Coron} and \cite[Ch.\ 7]{Jurdjevicbook} for further details.

Let $A,B$ be $n\times n$ matrices, with $B \geq 0$ and symmetric. Letting $k\leq n$ be the rank of $B$, there exist $b_1,\dots,b_k\in \R^n$, unique up to orthogonal transformations, such that $B = \sum_{i=1}^k b_i b_i^*$. Let also $Q$ be a symmetric $n\times n$ matrix, and $T>0$. We are interested in \emph{admissible trajectories}, namely  curves $x:[0,T]\to \mathbb{R}^n$ for which there exists a \emph{control} $u \in L^2([0,T],\mathbb{R}^k)$ such that\footnote{Our notation differs from the classical one, where usually $B$ denotes the $n\times k$ matrix whose columns are $b_1,\dots,b_k$ (as done also in \cite{BR-comparison}). With this notation \eqref{eq:lq1} becomes $\dot{x} = Ax + Bu$. 

To pass from the notation used in this paper to the classical one, one should replace $B$ with $BB^*$. The Kalman condition \eqref{eq:kalman} has the same expression with respect to either notation.}
\begin{equation}\label{eq:lq1}
\dot{x} = Ax + \sum_{i=1}^k u_i b_i.
\end{equation}
Thus, we look for admissible trajectories with fixed endpoints  $x(0) = x_0$,  $x(T) = x_1$, that minimize the quadratic functional $C_T: L^2([0,T],\mathbb{R}^k) \to \mathbb{R}$
\begin{equation}\label{eq:lq2}
C_T(u) = \frac{1}{2}\int_{0}^{T} \left(u^* u - x^*Qx \right)dt.
\end{equation}
Admissible trajectories minimizing \eqref{eq:lq2} are called \emph{minimizers}. The vector $Ax$ represents the \emph{drift}, while  $b_1,\dots,b_k$ are the \emph{controllable directions}. The matrix $Q$ is the \emph{potential} of the LQ problem.

We only deal with \emph{controllable} systems, i.e., there exists $m >0$ such that
\begin{equation}\label{eq:kalman}
\rank(B,AB,\ldots,A^{m-1}B) = n.
\end{equation}
Condition \eqref{eq:kalman} is known as \emph{Kalman condition} in control theory. It is equivalent to the fact that, for any choice of $x_0,x_1\in \R^{n}$ and $T>0$, there is a non-empty set of admissible trajectories $x :[0,T] \to \mathbb{R}^n$ joining $x_0$ with $x_1$.

It is well known that the admissible trajectories minimizing \eqref{eq:lq2} are projections $(p,x) \mapsto x$ of the solutions of the Hamilton equations
\begin{equation} \label{eq:eqhhh}
\dot{p}  = -\partial_x H, \qquad \dot{x} = \partial_p H, \qquad (p,x) \in T^*\R^n = \R^{2n},
\end{equation}
where the Hamiltonian function $H: \mathbb{R}^{2n} \to \mathbb{R}$ is defined by
\begin{equation}\label{eq:Hamiltonian}
H(p,x) = \frac{1}{2}\left( p^* B p +2 p^* A x + x^* Q x \right).
\end{equation}
Any LQ problem is uniquely determined by its Hamiltonian function, and vice-versa.

\begin{definition}
We say that $t_*>0$ is a conjugate time if there exists a non-trivial solution of the Hamilton equations \eqref{eq:eqhhh} such that $x(0) = x(t_*) = 0$.
\end{definition}

LQ problems either have no conjugate times, or an infinite and discrete set of them, depending on the Jordan normal form of the Hamiltonian system \cite{ARS-LQ}. The first  (i.e.,\ the smallest) conjugate time $t_c = t_c(A,B,Q)$ determines existence and the uniqueness of solutions of the LQ problem, as specified by the following proposition (see \cite[Sec.\ 16.4]{Agrachevbook}).

\begin{prop}\label{p:LQconj}
Let $t_c$ be the first conjugate time of the Hamiltonian \eqref{eq:Hamiltonian}, and consider the LQ problem \eqref{eq:lq1}-\eqref{eq:lq2}. Then, for any $x_0,x_1\in \R^n$,
\begin{itemize}
\item if $T<t_c$ there exists a unique minimizer connecting $x_0$ with $x_1$ in time $T$;
\item if $T>t_c$ there exists no minimizer connecting $x_0$ with $x_1$ in time $T$;
\item if $T= t_c$ existence of minimizers depends on $x_0,x_1$.
\end{itemize}
\end{prop}

The minimization of the functional \eqref{eq:lq2} with fixed endpoints and $T>0$ does not define a metric on $\R^{n}$, in general. Nevertheless, one can still define a distortion coefficient as follows. Fix $T=1$ in the LQ problem \eqref{eq:lq1}-\eqref{eq:lq2}. Furthermore, we assume throughout this section that $t_c >1$. This condition ensures existence and uniqueness of minimizers, and the well-posedness of the next definitions. This is not restrictive, since these assumptions will always be satisfied  for the cases we  consider. 

\begin{definition}
For $x_0,x_1 \in \R^n$ and $t\in[0,1]$, define
\[
Z_{t}(x_0,x_1)=\{x_{u}(t)\mid x_{u} : [0,1] \to \R^n  \text{ is the minimizer s.t. } x_u(0) = x_0,\, x_u(1) = x_1\}.
\]
\end{definition}

\begin{definition}\label{d:LQdist}
The distortion coefficient of the LQ problem \eqref{eq:lq1}-\eqref{eq:lq2} is
\begin{equation}
\beta_{t}^{A,B,Q}(x,y)=\limsup_{r\to 0}\frac{|Z_{t}(x,B_{r}(y))|}{|B_{r}(y)|}, \qquad t \in [0,1],
\end{equation}
where $x,y\in \R^n$, $B_r(y)$ denotes the Euclidean ball with center $y$ and radius $r>0$, and $|\cdot|$ denotes the Lebesgue measure of $\R^n$.
\end{definition}
As it will be clear from the proof of the next proposition, in Definition \ref{d:LQdist} one can replace the Euclidean ball with any set nicely shrinking for $r \to 0$.

\begin{prop}\label{p:LQdist}
The distortion coefficient of the LQ problem \eqref{eq:lq1}-\eqref{eq:lq2} does not depend on the choice of $x,y$, and satisfies
\begin{equation}\label{eq:firstrepformula}
\beta_{t}^{A,B,Q}=\frac{\det N(t)}{\det N(1)}>0, \qquad \forall t \in (0,1],
\end{equation}
where $M(t),N(t) :[0,1] \to \mathrm{Mat}(n \times n)$ are the solutions of the Hamiltonian system
\begin{equation}\label{eq:hamiltoniansystem}
\frac{d}{dt} \begin{pmatrix}
M \\
N
\end{pmatrix} = \begin{pmatrix}
-A^* & - Q \\
B &  A
\end{pmatrix} \begin{pmatrix}
M \\
N
\end{pmatrix}, \qquad \begin{pmatrix}
M(0) \\N(0) 
\end{pmatrix} = \begin{pmatrix}
\mathbbold{1} \\ \mathbbold{0}
\end{pmatrix}.
\end{equation}
Equivalently,  we have
\begin{equation}\label{eq:seconrepformula}
\beta_{t}^{A,B,Q}=\exp\left( -\int_{t}^{1}\tr(BV +A) ds\right) >0, \qquad \forall t \in (0,1],
\end{equation}
where $V :(0,1] \to \mathrm{Sym}(n \times n)$ is the solution of the matrix Riccati equation
\begin{equation}\label{eq:RiccatiLQ}
\dot{V} + A^*V + VA + VBV + Q = \mathbbold{0}, \qquad \lim_{t\to 0^+} V^{-1}(t) = \mathbbold{0}.
\end{equation}
\end{prop}
Notice that, under our assumptions, the Cauchy problem with limit initial datum \eqref{eq:RiccatiLQ} is well posed, and its solution is well-defined on $(0,1]$ (see Appendix~\ref{s:riccati}).
\begin{proof}
Fix $x \in \R^n$. Consider the map $E^t_x :\R^n \to \R^n$ that maps $p$ to the point $x(t)$ of the solution $(p(t),x(t))$ of the Hamilton equations with initial conditions $(p,x)$. We claim that $E^1_x$ is a smooth diffeomorphism and that $[0,1] \ni t\mapsto E^t_x(p)$ is the unique solution of the LQ problem \eqref{eq:lq1}-\eqref{eq:lq2} joining its endpoints.

Indeed, since  $t_c>1$ and by Proposition \ref{p:LQconj}, $E^1_x$ is surjective. Suppose that $E^1_x(p) = E^1_x(p')$. Let $(p(t),x(t))$ and $(p'(t),x'(t))$ the corresponding solutions of the Hamilton equations with initial conditions $(p,x)$ and $(p',x)$, respectively. By linearity of Hamilton equations for LQ problems, the difference $(p''(t),x''(t)):=(p'(t)-p(t),x'(t)-x(t))$ is still a solution, and $x''(0) = x''(1) = 0$. Since $t_c >1$, such a solution must be trivial, and in particular $p=p'$.

Suppose now that $p$ is a critical point for $E^1_x$, in particular there exists $\dot{p}\neq 0 $ in $\R^n$ such that $E^1_x(p+ \varepsilon \dot{p}) = E^1_x(p) + o(\varepsilon)$. By linearity of Hamilton equations defining $E^1_x$, we obtain that $E^1_0(\dot{p}) = 0$. Since $1$ cannot be a conjugate time, $\dot{p} = 0$, and $E^1_x$ is a submersion (the same argument shows that $E^t_0$ is a submersion for all $t<t_c$).

This concludes the proof of the claim. In particular, for all $x,y \in \R^n$ with $y \neq x$, and $y = E_x^1(p)$, we have $Z_t(x,y) = E_x^t(p)$. It follows directly from the definition that
\begin{align}
\beta_t^{A,B,Q}(x,y) = \frac{\det D E^t_x(p)}{\det D E^1_x(p)}, \qquad t \in [0,1],
\end{align}
where $y = E_x^1(p)$, and $D$ denotes the differential. By definition of $E^t_x$, and the linearity of the LQ Hamilton's equation, we immediately see that the linear map $DE^t_x(p)$, which we identify with a matrix $N(t)$ in coordinates, is the solution of
\begin{equation}
\frac{d}{dt} \begin{pmatrix}
M \\
N
\end{pmatrix} = \begin{pmatrix}
-A^* & - Q \\
B &  A
\end{pmatrix} \begin{pmatrix}
M \\
N
\end{pmatrix}, \qquad \begin{pmatrix}
M(0) \\N(0) 
\end{pmatrix} = \begin{pmatrix}
\mathbbold{1} \\ \mathbbold{0}
\end{pmatrix},
\end{equation}
thus proving the first representation formula \eqref{eq:firstrepformula}. This also proves that the distortion coefficient $\beta_t^{A,B,Q}(x,y)$ does not depend on $x$ and $y$.

Notice that $\beta_t^{A,B,Q} >0$ for all $t \in (0,1]$ (otherwise this would imply the existence of a conjugate time at $t<1$), hence  $\det N(t) > 0$ for $t \in (0,1]$. 

To prove the second representation formula \eqref{eq:seconrepformula}, for all $t>0$ we have
\begin{align}
\beta_t^{A,B,Q} & = \exp\left( -\int_t^1 \frac{d}{ds} \log \det N ds \right) 
\\
& 
=\exp\left( -\int_t^1 \tr (\dot{N} N^{-1}) ds \right)
\\
& 
=\exp\left( -\int_t^1 \tr (BV + A) ds \right),
\end{align}
where, in the last passage, we defined $V(t)=M(t)N(t)^{-1}$ for all $t \in (0,1]$, and we used the Hamiltonian system. A straightforward computation shows that $V$ satisfies the Riccati equation \eqref{eq:RiccatiLQ}. Furthermore, since $N(0)=\mathbbold{0}$ and $M(0) =\mathbbold{1}$, we have
\begin{equation}
\lim_{t \to 0^+} V^{-1}(t) = \lim_{t \to 0^+} N(t) M(t)^{-1} = \mathbbold{0},
\end{equation}
concluding the proof.
\end{proof}
We will need the following homogeneity property.
\begin{lemma} \label{l:betaomo}
For every $\eps>0$, it holds $\beta_{t}^{A,\eps B,Q}=\beta_{t}^{A, B,\eps Q}$.
\end{lemma}
\begin{proof}
It is enough to check that if $(M(t),N(t))$ is a solution of
\begin{equation}
\frac{d}{dt} \begin{pmatrix}
M \\
N
\end{pmatrix} = \begin{pmatrix}
-A^* & - \eps Q \\
B &  A
\end{pmatrix} \begin{pmatrix}
M \\
N
\end{pmatrix}, \qquad \begin{pmatrix}
M(0) \\N(0) 
\end{pmatrix} = \begin{pmatrix}
\mathbbold{1} \\ \mathbbold{0}
\end{pmatrix},
\end{equation}
then the pair $(\overline M(t),\overline N(t)):=(M(t),\eps N(t))$ satisfies
\begin{equation}
\frac{d}{dt} \begin{pmatrix}
M \\
N
\end{pmatrix} = \begin{pmatrix}
-A^* & - Q \\
\eps B &  A
\end{pmatrix} \begin{pmatrix}
M \\
N
\end{pmatrix}, \qquad \begin{pmatrix}
M(0) \\N(0) 
\end{pmatrix} = \begin{pmatrix}
\mathbbold{1} \\ \mathbbold{0}
\end{pmatrix}.\qedhere
\end{equation}
\end{proof}

\section{Constant curvature models} \label{s:somemodels}
Let $\y$ be the Young diagram associated with an ample, equiregular geodesic, and let $\Gamma_1 = \Gamma_1(\y)$, $\Gamma_2 = \Gamma_2(\y)$ be the $n\times n$ matrices defined in Section~\ref{s:matrices}. Let $Q$ be a symmetric $n\times n$ matrix.

\begin{definition}
We denote by $\mathrm{LQ}(\y;Q)$ the \emph{constant curvature model}, associated with a Young diagram $\y$ and constant curvature equal to $Q$, defined by the LQ problem with Hamiltonian
\begin{equation}
H(p,x) = \frac{1}{2}\left(p^* B p + 2p^*Ax + x^*Qx\right), \qquad A = \Gamma_1^*(\y), \quad B = \Gamma_2(\y).
\end{equation}
We denote by $\beta_t^{\y,Q}=\beta_t^{A,B,Q}$ its distortion coefficient.
\end{definition}

In the rest of this section we use Proposition \ref{p:LQdist} to provide several examples of distortion coefficients. We start by recovering,  within our framework, the usual Riemannian ones.

\subsection{The Riemannian case}\label{ex:1}

Let $\y$ be a Young diagram with a single column of length $n$ (which is the case for a Riemannian geodesic). We have $A=\mathbbold{0}_n$, $B=\mathbbold{1}_n$. Let also $Q = \k \mathbbold{1}_n$(we drop the subscript since the dimension is fixed). In this case the Hamiltonian of the corresponding LQ problem is
\begin{equation}
H(p,x) = \frac{1}{2} \left(|p|^2 +\k |x|^2\right),
\end{equation}
which is the Hamiltonian of a harmonic oscillator (for $\k >0$), a free particle (for $\k =0$) or a harmonic repulsor (for $\k <0$). The system
\begin{equation}
\frac{d}{dt} \begin{pmatrix}
M \\
N
\end{pmatrix} = \begin{pmatrix}
\mathbbold{0} & - \k \mathbbold{1} \\
\mathbbold{1} &  \mathbbold{0}
\end{pmatrix} \begin{pmatrix}
M \\
N
\end{pmatrix}, \qquad \begin{pmatrix}
M(0) \\N(0) 
\end{pmatrix} = \begin{pmatrix}
\mathbbold{1} \\ \mathbbold{0}
\end{pmatrix},
\end{equation}
is equivalent to the second order equation $\ddot{N} + \k N = \mathbbold{0}$, with $N(0)=\mathbbold{0}$ and $\dot N(0)=\mathbbold{1}$. We get in this case
\begin{equation}
\beta_{t}^{\y,Q}=\frac{\det N(t)}{\det N(1)}=
\begin{cases}
\left(\frac{\sin (\alpha t)}{\sin (\alpha)} \right)^{n} &\k>0,\\
t^{n}& \k=0,\\
\left(\frac{\sinh (\alpha t)}{\sinh(\alpha)} \right)^{n} &\k<0,
\end{cases}\qquad \alpha:=\sqrt{|\k|}.
\end{equation}
We will adopt a unified notation for the coefficient by writing
\begin{equation}
\beta_{t}^{\y,Q}=
\left(\frac{\sin (\sqrt{\k} t)}{\sin (\sqrt{\k})} \right)^{n}, \qquad \k \in \R,
\end{equation}
where we regard the above as an analytic function of $\k$, choosing the principal branch of the square root on the complex plane.

If we choose, instead, $Q=\mathrm{diag}\{\k,\ldots,\k,0\}$ we get
\begin{equation}
\beta_{t}^{\y,Q}=
t\left(\frac{\sin (\sqrt{\k} t)}{\sin (\sqrt{\k})} \right)^{n-1}, \qquad \k \in \R,
\end{equation}
which is the sharp Riemannian model coefficient of Theorem \ref{t:comp1}.

\subsection{The two-columns case} \label{ex:2}

Let $\y$ be a Young diagram with a single row of length $2$, and let $Q = \diag\{\k_1,\k_2\}$, with $\k_1,\k_2 \in \R$. In this case
\begin{equation}\label{eq:ABQ2}
A = \begin{pmatrix}
0 & 0 \\
1 & 0
\end{pmatrix}, \qquad B= \begin{pmatrix}
1 & 0 \\
0 & 0
\end{pmatrix}, \qquad Q= \begin{pmatrix}
\k_1 & 0 \\
0 & \k_2
\end{pmatrix}.
\end{equation}
The Hamiltonian of $\mathrm{LQ}(\y;Q)$ is
\begin{equation}
H(p,x) = \frac{1}{2} \left(p_1^2 +2p_{2}x_{1}+\k_{1} x_{1}^2+\k_{2}x_{2}^{2}\right).
\end{equation}
We can compute through Proposition \ref{p:LQdist} the distortion coefficient. By reduction to Jordan normal form of the corresponding Hamiltonian system \eqref{eq:hamiltoniansystem} (see details in \cite[Prop.\ 28]{RS17}), one obtains 
\begin{equation}
\det N(t)=\frac{\theta _-^2 \sin ^2\left(\theta _+ t\right)-\theta _+^2 \sin ^2\left(\theta _- t\right)}{4 \theta _-^2 \theta _+^2(\theta _-^2- \theta _+^2)},
\end{equation} 
where, choosing the principal branch of the square root, we set
\begin{equation}\label{eq:set}
\theta_{\pm}= \frac{1}{2}(\sqrt{x+y} \pm \sqrt{x-y}),\qquad \text{with} \qquad x = \frac{\k_1}{2}, \quad y = \frac{\sqrt{4\k_2 + \k_1^2}}{2}.
\end{equation}
Thus the distortion coefficient is
\begin{equation}\label{eq:stima}
\beta_{t}^{\y,Q}=\frac{\theta _-^2 \sin ^2\left(\theta _+ t\right)-\theta _+^2 \sin ^2\left(\theta _- t\right)}{\theta _-^2 \sin ^2\left(\theta _+ \right)-\theta _+^2 \sin ^2\left(\theta _- \right)}.
\end{equation}
Notice that $\beta_t^{\y,Q}$ is understood as a real-analytic function of $\theta_\pm \in \mathbb{C}$.

\subsubsection{The case $\k_2=0$}\label{s:subcase1}
A particular two-columns case is obtained for $\k_2=0$ (e.g., it occurs in the Heisenberg group and, more in general, in Sasakian contact structures with bounded Tanaka-Webster curvature, cf. Section~\ref{s:applications}). Depending on the sign of $\k_1$, we set $\theta = \theta_+ = \pm \theta_-$. Then, the distortion coefficient \eqref{eq:stima} reduces to
\begin{equation}\label{eq:stima76}
\beta_{t}^{\y,Q}=\frac{\sin(t\theta)}{\sin (\theta)} \frac{t\theta \cos(t\theta)-\sin(t\theta)}{\theta \cos(\theta)-\sin(\theta)}, \qquad \theta = \frac{\sqrt{\k_1}}{2},
\end{equation}
where the right hand side is understood as a real-analytic function of $\theta \in \mathbb{C}$. For instance, if $\kappa_{1}=0$, then \eqref{eq:stima76}  reads $\beta_{t}^{\y,Q}=t^{4}$.

\subsubsection{The case $4\k_2 + \k_{1}^{2}=0$}\label{s:subcase2}
Another relevant case is  obtained when $4\k_2 + \k_{1}^{2}=0$. It will be important for the proof of Theorem \ref{t:mcp}. In this case $\theta_{-}=0$ in \eqref{eq:set}. One gets (replacing $\theta_{+}$ by $\theta$ in the above notation)
\begin{equation}\label{eq:stima1}
\beta_{t}^{\y,Q}=\frac{ \sin ^2\left(t\theta \right)-(t\theta)^2}{\sin ^2\left(\theta \right)-\theta^2}, \qquad  \theta = \frac{\sqrt{\k_1}}{2},
\end{equation}
where the right hand side is understood as a real-analytic function of $\theta \in \mathbb{C}$. 

\begin{lemma}\label{l:subcase2}
Let $\beta_t^{\y,Q}$ as in \eqref{eq:stima1}, and assume that $\k_1 \geq 0$. Then $\beta_t^{\y,Q}/t^4$ is a non-increasing function of $t\in (0,1]$. In particular
\begin{equation}\label{eq:stima2}
\beta_{t}^{\y,Q}\geq t^{4}, \qquad \forall t \in [0,1].
\end{equation}
The exponent $4$ is optimal, i.e.\ it cannot be replaced to a smaller one.
\end{lemma}
\begin{proof}
One can check that $\beta_t^{\y,Q} \sim t^4$ as $t\to 0$, giving the optimality part of the statement. It is sufficient to prove that, for all $t\in(0,1]$, it holds
\begin{equation}\label{eq:stima3}
t\frac{d}{dt}\log \beta_{t}^{\y,Q} \leq 4.
\end{equation}
After some manipulations, the left hand side of \eqref{eq:stima3} is rewritten as 
\begin{equation}
t\frac{d}{dt}\log \beta_{t}^{\y,Q} = \frac{2z(z-\sin z)}{z^2+2\cos z-2},
\end{equation}
where we have set $z=2t\theta$. We then show that 
for every real $z\ge 0$ we have
	\begin{equation}\label{eq:ineq-theta00}
	g(z):=\frac{2z(z-\sin z)}{z^2+2\cos z-2}\le 4,
	\end{equation}
	and that the equality holds if and only if $z=0$. Notice that $\lim_{z\rightarrow 0}g(z)=4$ and $z^2+2\cos z-2>0$ for $z>0$. To prove (\ref{eq:ineq-theta00}) it is enough to show that
	\begin{equation}
	\frac{2z(z-\sin z)}{z^2+2\cos z-2}< 4, \quad \text{for } z>0,
	\end{equation}
	or equivalently $f(z):=z^2+z\sin z+4\cos z-4>0$, for $z>0$.
	We have that $f(0)=f'(0)=f''(0)=0$ and $f'''(z)=\sin z -z\cos z>0$ if $0<z\le\tfrac{5\pi}{4}$. Hence $f(z)>0$ if $0<z\le\tfrac{5\pi}{4}$.
	Now assume $z>\tfrac{5\pi}{4}$ and observe that
	\begin{equation}
	f(z) = \left(z+\frac{\sin z}{2}\right)^2 + \left(4+\frac{\cos z}{2}\right)^2- \frac{81}{4} > \left(\frac{5\pi}{4}-\frac{1}{2}\right)^2 + \left(4-\frac{1}{2}\right)^2- \frac{81}{4}>0.	
	\end{equation}
	The proof is concluded.
\end{proof}

\section{Comparison of the distortion coefficient}\label{s:proofs}

The following result on the computation of the distortion coefficient $\beta_{t}(x,y)$ on a sub-Riemannian manifold is crucial. An equivalent statement is \cite[Lemma~44]{BR-G1}.

\begin{lemma}\label{l:distortioncomputed}
Let $x,y \in M$, with $y \notin \cut(x)$, and assume that the geodesic $\gamma:[0,1]\to  M$ joining $x$ with $y$ is ample and equiregular, with Young diagram $\y$. Let $X_1(t),\dots,X_n(t)$ be a canonical moving frame along $\gamma$  (cf.\ Appendix \ref{s:prel}). Then
\begin{equation}
\frac{d}{dt}\log\beta_t(x,y)= \tr(B V(t) +A) + \rho_{\mis,\g}(t), \qquad \forall t \in (0,1],
\end{equation}
where $V : (0,1]\to \mathrm{Sym}(n\times n)$ is the solution of the Riccati equation
\begin{equation}\label{eq:riccatiproof}
\dot{V} + A^*V + VA + VBV + R(t) = \mathbbold{0}, \qquad \lim_{t\to 0^+}V(t)^{-1} = \mathbbold{0}.
\end{equation}
Here $R_{ij}(t) = \mathfrak{R}_{\gamma}(t)(X_i(t),X_j(t))$, and $A= \Gamma_1^*(\y)$, $B=\Gamma_2(\y)$ are the normal form matrices defined in Appendix \ref{s:prel}.
\end{lemma}
\begin{proof}
Let $\lambda_0$ be the initial covector of the unique minimizing geodesic such that $\exp_x(\lambda_0) = y$. Since $y \notin \cut(x)$, there exists an open neighbourhood $\mathcal{O}$ of $y$ and $O \subset T_x^*M$ such that $\exp_x: O \to \mathcal{O}$ is a smooth diffeomorphism, and for all $\lambda' \in O$, the geodesic $t \mapsto \exp_x(t \lambda')$ is the unique minimizing geodesic joining $x$ with $y' = \exp_x(\lambda')$, and $y'$ is not conjugate with $x$ along such a geodesic. Assuming $r$ sufficiently small such that $\mathcal{B}_r(y) \subset \mathcal{O}$, let $A_r \subset O$ be the relatively compact set such that $\exp_x(A_r) = \mathcal{B}_r(y)$. The map $\exp^t_x(\cdot) = \exp_x(t\cdot)$ is a smooth diffeomorphism from $A_r$ onto $Z_t(x,\mathcal{B}_r(y))$. In particular, we have
\begin{equation}\label{eq:differentiation}
\beta_t(x,y)  = \lim_{r \downarrow 0} \frac{\int_{A_r} \exp_x^{t*} \mis}{\int_{A_r} \exp_x^{1*} \mis} = \frac{(\exp_x^{t*} \mis)(\lambda_0)}{(\exp_x^{1*} \mis)(\lambda_0)}.
\end{equation}
The right hand side of \eqref{eq:differentiation} is the ratio of two smooth tensor densities computed at $\lambda_0$. To compute it, we evaluate both factors on a $n$-tuple of independent vectors of $T_x^*M$. Thus, pick a Darboux frame $E_1(t),\dots,E_n(t),F_1(t),\dots,F_n(t) \in T_{\lambda(t)}(T^*M)$ such that $\pi_* E_i(t) = 0$ and $\pi_* F_i(t) = X_i(t)$ for all $t \in [0,1]$, $i=1,\dots,n$. Then,
\begin{equation}
(\exp_x^{t*} \mis)(E_1(0),\dots,E_n(0)) = \mis(\pi_* \circ e^{t\vec{H}}_* E_1(0),\dots,\pi_* \circ e^{t\vec{H}}_* E_n(0)).
\end{equation}
The $n$-tuple $\J_i(t)=e^{t\vec{H}}_* E_i(0)$, $i=1,\dots,n$, can be written as
\begin{equation}
\J_i(t) = \sum_{j=1}^n E_j(t) M_{ji}(t) + F_j(t) N_{ji}(t), \qquad \forall i=1,\dots,n,
\end{equation}
for some smooth families of $n\times n$ matrices $M(t),N(t)$, such that $M(0) = \mathbbold{1}$ and $N(0) = \mathbbold{0}$.
Therefore we have
\begin{equation}
\beta_t(x,y) = \frac{\det N(t)}{\det N(1)} \frac{\mis(X_1(t),\dots,X_n(t))}{\mis(X_1(1),\dots,X_n(1))}.
\end{equation}
Since $\gamma(t)$ is not conjugate to $\gamma(0)$ for $t \in (0,1]$, we have $\beta_t(x,y)>0$ on that interval or, equivalently, $N(t)$ is non-degenerate for all $t\in (0,1]$. In particular, we have
\begin{align}
\frac{d}{dt }\log\beta_t(x,y)&  = \frac{d}{d t} \log\det N(t) + \frac{d}{dt} \log \mis (X_1(t),\dots,X_n(t)) \\
& = \tr(\dot{N}(t) N(t)^{-1}) + \frac{d}{dt} \log \mis (X_1(t),\dots,X_n(t)).
\end{align}
Recall that, by Lemma \ref{l:subspacesriccati}, the pair $M(t),N(t)$ satisfies
\begin{equation}
\frac{d}{dt} \begin{pmatrix}
M \\
N
\end{pmatrix} = \begin{pmatrix}
-A(t)^* & - R(t) \\
B(t) & A(t)
\end{pmatrix} \begin{pmatrix}
M \\
N
\end{pmatrix},
\end{equation}
for some matrices $A(t)$, $B(t) \geq 0$, and $R(t) = R(t)^*$. Thus
\begin{equation}\label{eq:lastequation}
\frac{d}{dt }\log\beta_t(x,y) = \tr (B(t) V(t) + A(t)) +\frac{d}{dt} \log \mis (X_1(t),\dots,X_n(t)),
\end{equation}
where $V(t) = M(t)N(t)^{-1}$, which is well defined on $(0,1]$, satisfies 
\begin{equation}
\dot{V}(t) + A(t)^*V +V A(t) + V B(t) V + R(t) = \mathbbold{0}, \qquad \lim_{t\to 0^+} V(t)^{-1} = \mathbbold{0}.
\end{equation}
We conclude the proof by choosing $E_1(t),\dots,E_n(t),F_{1}(t),\dots, F_{n}(t)$ to be a canonical Darboux frame. In this case $A(t) = \Gamma_1^*(\y)$, $B(t) = \Gamma_2(\y)$ appearing in \eqref{eq:lastequation} are in the normal form as described in Appendix \ref{s:prel}, and $R_{ij}(t) =\mathfrak{R}_{\gamma}(t)(X_i(t),X_j(t))$. Finally, the second term in the r.h.s.\ of \eqref{eq:lastequation} is equal to $ \rho_{\mis,\g}(t)$, by definition.
\end{proof}

\subsection{Proof of Theorem \ref{t:srbe}} Assume that $R(t) \geq Q$ for a constant quadratic form $Q \in \mathrm{Sym}(n\times n)$. By the comparison theory for the matrix Riccati equation with limit initial datum (see Appendix \ref{s:riccati}), it follows that
\begin{equation}
V(t) \leq V^{\y,Q}(t), \qquad \forall t \in (0,1],
\end{equation}
where $V^{\y,Q}(t)$ is the unique solution of \eqref{eq:riccatiproof} with $R(t)$ replaced by $Q$. Using the formulas provided in Proposition \ref{p:LQdist} and Lemma \ref{l:distortioncomputed}, this implies
\begin{equation}\label{eq:dlogineq}
\frac{d}{dt}\log \beta_t(x,y) \leq  \frac{d}{dt}\log \beta_t^{\y,Q} + \rho_{\mis,\g}(t) , \qquad \forall t \in (0,1].
\end{equation}
We remark that the r.h.s.\ of the above equation would be $-\infty$ in presence of a conjugate time $t_* \in (0,1]$ of the LQ problem, which would give a contradiction to the smoothness of $\beta_t(x,y)$. Hence the first conjugate time of the LQ model must satisfy $t_c >1$, and $t \mapsto \beta_t^{\y,Q}$ is well defined, positive and smooth for all $t\in (0,1]$.

If $\rho_{\mis,\g}(t) \leq c$, then \eqref{eq:dlogineq} is equivalent to the fact that $t\mapsto e^{- ct}\beta_t(x,y)/\beta_t^{\y,Q}$ is non-increasing on $[0,1]$. Since $\beta_1(x,y) = \beta_1^{\y,Q}=1$, this implies $\beta_t(x,y) \geq \beta_t^{\y,Q}e^{c(t-1)}$.  The proof is similar assuming reversed inequalities $\mathfrak{R}_{\gamma}(t) \leq Q$ and $\rho_{\mis,\g}(t) \geq c$. \hfill \qedsymbol

\subsection{Proof of Theorem \ref{t:srbe2}}
By Lemma \ref{l:distortioncomputed} we have
\begin{equation}\label{eq:piurho2}
\frac{d}{dt}\log\beta_t(x,y)= \tr(B V(t)  + A) + \rho_{\mis,\g}(t), \qquad \forall t \in (0,1].
\end{equation}
Here $A = \Gamma_1^*(\y)$ and $B =\Gamma_2(\y)$ are the matrices defined in Appendix \ref{s:prel}. In the previous expression we can omit $A$, since $\tr(A)=0$. Since $\tr(B)=k$, we have
\begin{align}\label{eq:piurho3}
\tr(BV(t))+\rho_{\mis,\g}(t) =\tr(BV(t))+\tr\left(\frac{\rho_{\mis,\g}(t)}{k}B\right)=\tr\left(B \Vr \right),
\end{align}
where we have set (recall that for our choice $B^{2}=B$)
\begin{equation}
\Vr(t):=V(t)+\frac{\rho_{\mis,\g}(t)}{k}B.
\end{equation}
Notice that $\Vr$ is invertible for small $t$ and $\lim_{t\to 0}\Vr(t)^{-1}=\mathbbold{0}$. This is a consequence of the fact that $\lim_{t\to 0^+} V(t)^{-1} = \mathbbold{0}$, and the identity
\begin{equation}
\Vr(t)=V(t)\left(\mathbbold{1}+\frac{ \rho_{\mis,\g}(t)}{k}V(t)^{-1}B\right).
\end{equation}
Using in a crucial way that $A^{*}B=BA=\mathbbold{0}$, we see that $\Vr$ satisfies
\begin{equation}\label{eq:riccati2}
\dot{\Vr}  + A^{*} \Vr + \Vr A + \Vr B \Vr+\overline R(t) = \mathbbold{0}, \qquad \lim_{t\to 0^+} \Vr(t)^{-1} = \mathbbold{0},
\end{equation}
where we defined
\begin{equation}\label{eq:riccatiR}
\overline R(t):=R(t)-\frac{\dot \rho_{\mis,\g}(t)}{k}B-\frac{ \rho_{\mis,\g}(t)}{k}\left(\Vr(t) B+B\Vr(t)\right)+\frac{ \rho^{2}_{\mis,\g}(t)}{k^{2}}B.
\end{equation}
Notice that $\overline R$ contains a term depending on $\Vr$. In order to use the Riccati comparison theory described in Appendix \ref{s:riccati} to control $\Vr$, we need to bound $\overline R$ uniformly with respect to $\Vr$. To do it, one pays a price on the coefficient of the quadratic term of \eqref{eq:riccati2}. This fact is formalized in the next lemma.
\begin{lemma}\label{l:piurho}
For every $N>n$ let us define
\begin{equation}
\Rr:=R(t)-\left(\frac{\dot \rho_{\mis,\g}(t)}{k}+\frac{\rho^{2}_{\mis,\g}(t)}{k^{2}}\frac{n}{N-n}\right)B, \qquad \overline{B}:= \frac{n}{N}B, \qquad \overline{A} = A.
\end{equation}
Then $\Vr(t)$ satisfies the following matrix Riccati inequality
\begin{equation}\label{eq:riccati3}
\dot{\Vr}  + \overline{A}^{*} \Vr + \Vr \overline{A} + \Vr \overline{B} \Vr+\Rr \leq \mathbbold{0}, \qquad \lim_{t\to 0^+} \Vr(t)^{-1} = \mathbbold{0}.
\end{equation}
\end{lemma}
\begin{proof}[Proof of Lemma \ref{l:piurho}]
Let $a>1$ such that $1-\frac{1}{a^2} = \frac{n}{N}$. Recalling that $B^2=B$, and omitting the dependence on $t$, we have
\begin{equation}
\left(\frac{ \rho_{\mis,\g}}{k} a B-\frac{1}{a}B\Vr \right)^{*}\left(\frac{ \rho_{\mis,\g}}{k}a B-\frac{1}{a}B\Vr\right)=
\frac{ \rho^{2}_{\mis,\g}}{k^{2}}a^{2}B+\frac{1}{a^{2}}\Vr B\Vr -\frac{\rho_{\mis,\g}}{k}(\Vr B + B \Vr).
\end{equation}
The left hand side of the above is non-negative, hence 
\begin{equation}\label{eq:lltrasp2}
-\frac{\rho_{\mis,\g}}{k}(\Vr B + B \Vr) \geq 
-\frac{ \rho^{2}_{\mis,\g}}{k^{2}}a^{2}B-\frac{1}{a^{2}}\Vr B\Vr. 
\end{equation}
Replacing \eqref{eq:riccatiR} in the last term of \eqref{eq:riccati2} we obtain
\begin{equation}
\dot{\Vr}  + A^{*} \Vr + \Vr A +\left(1-\frac{1}{a^{2}}\right) \Vr B \Vr+R(t)-\frac{\dot \rho_{\mis,\g}(t)}{k}B+ (1-a^{2})\frac{ \rho^{2}_{\mis,\g}}{k^{2}}B\leq \mathbbold{0},
\end{equation}
hence the conclusion using that $1-a^2 = -\frac{n}{N-n}$ by our choice of $a$.
\end{proof}
Combining \eqref{eq:piurho2} and \eqref{eq:piurho3} we get
\begin{align}\label{eq:piurho4}
\frac{1}{N}\frac{d}{dt}\log\beta_{t}(x,y)=\frac{1}{N} \tr(B\Vr(t) + A) = \frac{1}{n} \tr(\overline{B} \Vr(t) + \overline{A}).
\end{align}
The assumption on the Bakry-Émery curvature means precisely that
\begin{equation}
\Rr \geq \frac{N}{n}Q = :\overline{Q}.
\end{equation}
Thus, by Lemma \ref{l:piurho} and Riccati comparison (see Appendix \ref{s:riccati}), we have 
\begin{equation}
\Vr(t)\leq V^{\overline{A},\overline{B},\overline{Q}}(t),
\end{equation}
where the latter is the solution of the Riccati equation associated with the LQ problem defined by $\overline{A},\overline{B}$ and $\overline{Q}$. It follows by \eqref{eq:piurho4} and Proposition \ref{p:LQdist} that
\begin{equation}\label{eq:lasteqBE}
\frac{1}{N}\frac{d}{dt}\log\beta_{t}(x,y) \leq \frac{1}{n} \frac{d}{dt}\log\beta_t^{\overline{A},\overline{B},\overline{Q}} = \frac{1}{n}\frac{d}{dt}\log\beta_t^{\y,Q}, 
\end{equation}
where, in the last equality, we used the definitions of $\overline{A},\overline{B},\overline{Q}$ and Lemma \ref{l:betaomo}.
Equation \eqref{eq:lasteqBE} is equivalent to the fact that the weighted ratio $\beta_t(x,y)^{1/N}/(\beta_t^{\y,Q})^{1/n}$ is a non-increasing function of $t$, and in particular $\beta_t(x,y)^{1/N} \geq (\beta_t^{\y,Q})^{1/n}$. \hfill \qedsymbol
\section{Ricci curvature type comparison}\label{s:proofsricci}

By Lemma \ref{l:distortioncomputed}, the distortion coefficient can be computed by solving a matrix Riccati equation. Let $\gamma$ be a geodesic on a $n$-dimensional Riemannian manifold $M$. In this case $X_1(t),\dots,X_n(t)$ are a canonical moving frame along $\gamma$ if and only if they are a parallel orthonormal frame (see Appendix \ref{s:prel}). In this case $A=\mathbbold{0}_n$, $B=\mathbbold{1}_n$, and the Riccati equation is simply
\begin{equation}\label{eq:riccati4}
\dot{V} + V^{2} +R(t)=\mathbbold{0}, \qquad  R_{ij}(t) = R_g(\dot\gamma(t),X_i(t),X_j(t),\dot\gamma(t)),
\end{equation}
where $R_g$ is the Riemann curvature tensor. Taking the trace of \eqref{eq:riccati4}, and using the Cauchy-Schwartz inequality, one shows that $v:=\tfrac{1}{n}\tr V$ satisfies
\begin{equation}\label{eq:riccati7}
\dot v+  v^{2} +\frac{r}{n} \leq 0, \qquad r(t) = \tr R(t).
\end{equation}

Notice that \eqref{eq:riccati7} is a scalar inequality, and it is simpler to handle with respect to \eqref{eq:riccati4}. Since in the Riemannian case $\frac{1}{n}\frac{d}{dt}\log \beta_t(x,y)= v$, one can prove directly from \eqref{eq:riccati7} comparison theorems for the distortion coefficient under Ricci lower curvature bounds. The same argument applies to the case of weighted Riemannian manifolds, replacing the Ricci curvature with the classical Bakry-Émery one.

\medskip
In the general sub-Riemannian setting this argument does not work. Recall that, by Lemma \ref{l:distortioncomputed}, the logarithmic derivative of $\beta_t$ is given by $\tr(BV(t)+A)=\tr(BV(t))$ (recall that $\tr A=0$), where $V(t)$ solves the general matrix Riccati equation \eqref{eq:riccatiproof}. In contrast with the Riemannian case the latter does not yield, upon tracing, a scalar differential inequality for $\tr(BV)$. It turns out that different  sets of tangent directions  along $\gamma$  behave differently, according to the structure of the Young diagram $\y$. However, we are able to trace among the directions corresponding to the rows of $\y$ that have the same length, namely rows in the same level. The proof of Theorems~\ref{t:srbericci1} and \ref{t:srbericci2} is based on the following two steps.

\medskip
\textbf{Splitting:} We split the matrix Riccati equation 
\begin{equation}
\dot{V}  + A^{*} V + V A  + V B V + R(t)= \mathbbold{0}, \qquad \lim_{t\to 0^+}V(t)^{-1} = \mathbbold{0},
\end{equation}
in several, lower-dimensional equations for special diagonal blocks of $V(t)$. In these equations, only some blocks of $R(t)$ do appear. We obtain one Riccati equation for each row of the Young diagram, of dimension equal to the length of the row. 

\medskip
\textbf{Tracing:} after the splitting step, we sum the Riccati equations corresponding to the rows with the same length, since all these equations are, in some sense, compatible (they have the same $A,B$ matrices). We obtain one Riccati equation for each level of the Young diagram, of dimension equal to the length $\ell$ of the level. The curvature matrix is replaced by a diagonal matrix, whose diagonal elements are the Ricci curvatures of the superboxes $\lev_1,\ldots,\lev_\ell$ in the given level. 

\medskip
In the Riemannian case, this procedure leads to the single, scalar Riccati inequality \eqref{eq:riccati7}, since there is only one level of length one, and a single Ricci curvature.

\subsection{Proof of Theorem~\ref{t:srbericci1}}

 Consider the Riccati Cauchy problem with limit initial datum as in Lemma \ref{l:distortioncomputed}, whose unique maximal solution is symmetric and defined on a maximal interval $I \subseteq (0,+\infty)$ (cf.\ Lemma \ref{l:limit})
\begin{equation}\label{eq:riccatibeginning}
\dot{V}  + A^* V + V A  + V B V+ R(t) = \mathbbold{0}, \qquad \lim_{t\to 0^+}V(t)^{-1} = \mathbbold{0},
\end{equation}
where $A=\Gamma_1^*(\y)$ and $B=\Gamma_2(\y)$ are the $n\times n$ matrices associated with the Young diagram  $\y$ of $\gamma$ (cf.\ Appendix \ref{s:prel}). We label the components of $V(t)\in \mathrm{Sym}(n\times n)$ according to the boxes of the Young diagram. Regard then $V(t)$ as a block matrix, labelled as the boxes of the Young diagram (cf.\ Appendix \ref{s:young}). More precisely, let $a,b=1,\dots,k$ be the rows of $\y$, of length $n_a$ and $n_b$ respectively. The block $ab$ of $V(t)$, denoted $V_{ab}(t)$ is a $n_a\times n_b$ matrix with components $V_{ai,bj}(t)$, for $i=1,\ldots,n_a$, $j=1,\ldots,n_b$. Let us focus on the diagonal blocks
\begin{equation}
V(t) = \begin{pmatrix}
V_{11}(t)   &   & * \\
 &   \ddots &  \\
* &   & V_{kk}(t) 
\end{pmatrix}.
\end{equation}
The generic $a$-th block on the diagonal $V_{aa}(t)$ satisfies
\begin{equation}
\dot{V}_{aa} + \Gamma_1 V_{aa} + V_{aa}\Gamma_1^*  + V_{aa} \Gamma_2 V_{aa} + \wt{R}_{aa}(t) = \mathbbold{0},
\end{equation}
where
\begin{equation}\label{eq:boundRtilde}
\wt{R}_{aa}(t) = R_{aa}(t) + \sum_{b\neq a} V_{ab}(t) \Gamma_2 V_{ab}^*(t) \geq R_{aa}(t).
\end{equation}
Here $\Gamma_1 = \Gamma_1(\y_a)$, $\Gamma_2=\Gamma_2(\y_a) \geq 0$ are $n_a\times n_a$ matrices corresponding to the $a$-th row $\y_a$ of the Young diagram (see Section \ref{s:matrices}). Thanks to the ampleness assumption one can show that the block $V_{aa}$ satisfies (see \cite[Lemma 5.4]{BR-comparison})
\begin{equation}
\lim_{t \to 0^+} V_{aa}(t)^{-1} = \mathbbold{0}.
\end{equation}
Hence $V_{aa}(t)$ is solution of the Riccati matrix equation with limit initial data
\begin{equation}\label{eq:pretrace}
\dot{V}_{aa}+ \Gamma_1 V_{aa} + V_{aa}\Gamma_1^* + V_{aa} \Gamma_2 V_{aa} +\wt{R}_{aa}(t)  = \mathbbold{0}, \qquad \displaystyle \lim_{t\to 0^+} V_{aa}(t)^{-1} = \mathbbold{0}.
\end{equation}

We now proceed with the second step of the proof, namely tracing over the levels of the Young diagram. 
Let $a \in \{a_1,\ldots,a_r\}$ be the rows $\y_a$ in a given level $\lev$ (of size $r$), whose rows have length $\ell =n_a$. Define the $\ell \times \ell$ symmetric matrix:
\begin{equation}
V_\lev := \frac{1}{r}\sum_{a \in \lev} V_{aa}.
\end{equation}
Starting from \eqref{eq:pretrace} it is easy to see that $V_\lev(t)$ satisfies
\begin{equation}
\dot{V}_\lev +\Gamma_1 V_\lev + V_\lev\Gamma_1^*  + V_\lev \Gamma_2 V_\lev + R_\lev(t) = \mathbbold{0}, \qquad \lim_{t\to 0^+} V_\lev(t)^{-1} = \mathbbold{0},
\end{equation}
where  the $\ell \times \ell$ matrix $R_\lev(t)$ is defined by
\begin{equation}
\begin{split}
R_\lev(t) :=\, &\frac{1}{r}\sum_{a \in \lev} \wt{R}_{aa}(t) + \frac{1}{r}\sum_{a \in \lev} V_{aa}\Gamma_2V_{aa} - V_\lev \Gamma_2 V_\lev\\
=\, &\frac{1}{r}\sum_{a \in \lev} \wt{R}_{aa}(t) + \frac{1}{r}\left[\sum_{a \in \lev} (V_{aa}\Gamma_2)(V_{aa}\Gamma_2)^* - \frac{1}{r}\left(\sum_{a \in \lev} V_{aa}\Gamma_2\right)\left(\sum_{a \in \lev} V_{aa}\Gamma_2\right)^*\right].
\end{split}
\end{equation}
It turns out that, as a consequence of a non-trivial matrix version of the Cauchy-Schwarz inequality, the term in square bracket in the above equation is non-negative (see \cite[Lemma 5.5]{BR-comparison}). Hence combining the latter with \eqref{eq:boundRtilde} we have
\begin{equation}\label{eq:almostdone}
R_\lev(t) \geq \frac{1}{r} \sum_{a \in \lev} \wt{R}_{aa}(t) \geq \frac{1}{r}\sum_{a \in \lev} R_{aa}(t).
\end{equation}
 The matrix $R(t)$ is normal in the sense of Zelenko-Li (cf.\ Definition \ref{d:normal}). In particular $R_{ai,aj} \neq 0$ if and only if $i=j$. Thus $R_{aa}(t)$ is diagonal and we have
\begin{equation}
\sum_{a \in \lev} R_{aa}(t) = \sum_{a\in\lev}\begin{pmatrix}
R_{a1,a1}(t) &   & 0 \\
 & \ddots &  \\
0 &   & R_{a\ell,a\ell}(t) \\
\end{pmatrix}=
 \begin{pmatrix}
\Ric_{\gamma}^{\lev_1}(t) &   & 0 \\
 & \ddots &  \\
0 &   & \Ric_{\gamma}^{\lev_\ell}(t) \\
\end{pmatrix},
\end{equation}
where we used the definition of sub-Riemannian Ricci curvature corresponding to the level $\lev$. We have so far proved that, for any level $\lev$, the trace over the level $V_\lev(t)$ satisfies the $\ell \times \ell$ matrix Riccati equation
\begin{equation}\label{eq:riccatigigatraced}
\dot{V}_\lev + \Gamma_1 V_\lev + V_\lev \Gamma_1^*  + V_\lev \Gamma_2 V_\lev + R_\lev(t)= \mathbbold{0}, \qquad \lim_{t \to 0^+} V_\lev(t)^{-1} = \mathbbold{0},
\end{equation}
and, under our hypotheses, $R_\lev(t) \geq Q_\lev = \diag\{\k_{\lev_1},\ldots,\k_{\lev_\ell}\}$. Thus, by Riccati comparison, \eqref{eq:riccatigigatraced} implies that for any level $\lev$
\begin{equation}\label{eq:comparisonresult}
V_\lev(t) \leq V^{\y_\lev,Q_\lev}(t), \qquad \forall t \in I,
\end{equation}
where $\y_\lev$ is a Young diagram composed by a single row, of length $\ell=\ell_\lev$, and $Q_\lev = \diag\{\k_{\lev_1},\ldots,\k_{\lev_\ell}\}$. Thus, by Lemma \ref{l:distortioncomputed} and Proposition \ref{p:LQdist}, we obtain (we omit the trace-free term $A=\Gamma_1^*(\y)$ for simplicity)
\begin{align}
\frac{d}{d t}\log\beta_t(x,y) & = \tr(\Gamma_2(\y) V(t)) \\
& = \sum_{\lev}  \sum_{a \in \lev} \tr(\Gamma_2(\y_\lev) V_{aa}(t) ) \\
& = \sum_{\lev} r_\lev \tr(\Gamma_2(\y_\lev) V_\lev(t) )\\
& \leq \sum_{\lev} r_\lev \tr(\Gamma_2(\y_\lev) V^{\y_\lev,Q_\lev}(t) )\\
& =\sum_{\lev} r_\lev \frac{d}{dt}\log \beta_t^{\y_\lev,Q_\lev},
\end{align}
where we used the block-diagonal structure of $\Gamma_2(\y)$ (cf.\ Appendix \ref{s:matrices}), and the sum is over all levels $\lev$ of the Young diagram and over all rows $a$ belonging to the levels $\lev$. Furthermore, $\beta_t^{\y_\lev,Q_\lev}$ is the model distortion coefficient of a LQ model whose Young diagram is a single line of length equal to $\ell=\ell_\lev$ and $Q=\diag(\k_{\lev_1},\dots,\k_{\lev_\ell})$.

The above result means that  that the ratio $\beta_t(x,y)/\left(\prod_{\lev} \beta_t^{\y_\lev,Q_\lev}\right)^{r_\lev}$ is a non-increasing function of  $t\in (0,1]$, and in particular it is $\geq 1$.
\hfill \qedsymbol

\subsection{Proof of Theorem~\ref{t:srbericci2}}

We argue as in the proof of Theorem \ref{t:srbe2}. We consider, instead of the matrix $V(t)$ solution of \eqref{eq:riccatibeginning}, the matrix
\begin{equation}
\Vr(t) = V(t) + \frac{\rho_{\mis,\g}(t)}{k} B,
\end{equation}
that satisfies the matrix Riccati inequality
\begin{equation}
\dot{\Vr}  + \overline{A}^{*} \Vr + \Vr \overline{A} + \Vr \overline{B} \Vr+\Rr \leq \mathbbold{0}, \qquad \lim_{t\to 0^+} \Vr(t)^{-1} = \mathbbold{0},
\end{equation}
where
\begin{equation}
\Rr:=R(t)-\left(\frac{\dot \rho_{\mis,\g}(t)}{k}+\frac{n}{N-n}\frac{\rho^{2}_{\mis,\g}(t)}{k^{2}}\right)B, \qquad  \overline{B}:= \frac{n}{N}B, \qquad \overline{A} = A.
\end{equation}
The matrix $\Vr(t)$ is related to the distortion coefficient by the formula
\begin{equation}
\frac{1}{N}\frac{d}{dt}\log \beta_t(x,y) = \frac{1}{n} \tr(\overline{B} \Vr(t) +\overline{A}) =  \frac{1}{n} \tr(\overline{B} \Vr(t) +\overline{A}) .
\end{equation}
Using now the same technique as in the proof of Theorem \ref{t:srbericci1}, we obtain under the assumptions on the  sub-Riemannian Bakry-Émery Ricci curvature that the ratio
$\beta_t(x,y)^{1/N}/\left(\prod_{\lev} \beta_t^{\y_\lev,Q_\lev}\right)^{r_\lev/n}$ is a non-increasing function of $t\in (0,1]$, and in particular it is $\geq 1$. \hfill \qedsymbol

\subsection{Proof of Theorem \ref{t:mcp}}

By assumption $\rho_{\mis,\g} \leq 0$, and we can use Theorem \ref{t:srbericci1}. One should be careful, since for the latter we employ the general notation, while for Theorem \ref{t:mcp} we label the Ricci curvatures according to Figure \ref{f:YdMCP}. 

The Young diagram of $\gamma$ has two levels. For the Ricci curvatures of the first level, by our assumptions, it holds
\begin{equation}
\frac{1}{n-k}\mathfrak{Ric}_{\gamma}^b(t) \geq \k_b, \qquad \frac{1}{n-k}\mathfrak{Ric}_{\gamma}^a(t) \geq \k_a, \qquad \forall t \in [0,1],
\end{equation}
for some $\k_a,\k_b$ such that $\k_b \geq 0$ and $4\k_a+\k_b^2 \geq 0$. Up to reducing $\k_a$, and relabelling the constants, we can find $\k_1,\k_2 \in \R$ such that
\begin{equation}
\frac{1}{n-k}\mathfrak{Ric}_{\gamma}^b(t) \geq \k_1, \qquad \frac{1}{n-k}\mathfrak{Ric}_{\gamma}^a(t) \geq \k_2, \qquad \forall t \in [0,1],
\end{equation}
with $\k_1 \geq 0$ and $4\k_2+\k_1^2=0$. The corresponding LQ model, associated with a Young diagram of one line and two columns, and with $Q=\diag\{\k_1,\k_2\}$. Let us denote by $\beta_t^{\k_1,\k_2}$ the corresponding distortion coefficient, which is precisely the subcase discussed in Section \ref{s:subcase2}.

For the Ricci curvatures of the second level we have
\begin{equation}
\mathfrak{Ric}_{\gamma}^c(t) \geq 0, \qquad \forall t \in [0,1].
\end{equation}
The corresponding LQ model, associated with Young diagram of a single block, and $Q=0$, is the flat Riemannian one discussed in Section \ref{ex:1}, that is $\beta_t^0 = t$.

The comparison function of Theorem \ref{t:srbericci1} is the product  of two factors, one for each level, raised to the appropriate power depending on the size of the level ($n-k$ for the first level, and $2k-n$ for the second level, see Figure \ref{f:YdMCP}). We obtain that
\begin{equation}
\frac{\beta_t(x,y)}{(\beta_t^{\k_1,\k_2})^{n-k} (\beta_t^0)^{2k-n}} \text{ is a non-increasing function of $t\in(0,1]$}.
\end{equation}
As we already remarked $\beta_t^0=t$. Furthermore, by Lemma \ref{l:subcase2}, $\beta_t^{\k_1,\k_2}/ t^4$ is a non-increasing function of $t\in (0,1]$. We conclude that
\begin{equation}
\frac{\beta_t(x,y)}{t^{4(n-k)} \times t^{2k-n}} \text{ is a non-increasing function of $t\in(0,1]$}.
\end{equation}
In particular, since $\beta_1(x,y)=1$, we have that $\beta_t(x,y) \geq t^{k+3(n-k)}$ for all $t\in[0,1]$.

The exponent $k +3(n-k)$ is the smallest possible. This can be seen as follows. If $y\notin \cut(x)$ the asymptotics as $t\to 0$ of $\beta_t(x,y)$ is equal to the asymptotics of the Jacobian determinant of the sub-Riemannian exponential map $\exp_x : T_x^* M \to M$. If the geodesic $\gamma(t) = \exp_x(t\lambda)$ is ample and equiregular with Young diagram $\y$, this asymptotics is given by the geodesic dimension $\mathcal{N}_\lambda$ (see \cite[Lemma 6.27]{curvature}). If the Young diagram has two columns, then $\mathcal{N}_\lambda =k+3(n-k)$. \hfill \qedsymbol
\section{Applications}\label{s:applications}

In this section we apply our comparison results to the class of Sasakian manifolds (which contains the Heisenberg groups as a particular case), and 3-Sasakian manifolds. In both cases we provide formulas for the sub-Riemannian Ricci curvatures, written in terms of a suitable connection.

\subsection{Sasakian manifolds}\label{s:61}
We follow the notation of \cite{ABR-contact}, to which we refer to for details and references. A contact manifold $(M,\omega)$ is a smooth odd-dimensional manifold endowed with a 1-form such that $d\omega$ is non-degenerate on $\ker \omega$. We endow $\distr=\ker \omega$ with a sub-Riemannian metric $g$. The \emph{Reeb vector field} $X_{0}$ is the unique vector field satisfying $\omega(X_{0})=1$ and $d\omega(X_{0},\cdot)=0$. Since $X_0$ is transverse to $\distr$, we can extend $g$ to a Riemannian structure on $M$, by declaring $X_0$ to be unit and orthogonal to $\distr$. The contact endomorphism $J:TM\to TM$ is defined by:
\begin{equation}
g(X,JY)=d\omega(X,Y),\qquad \forall X,Y\in TM. 
\end{equation}
We always assume that $J$ is an almost-complex structure on $\distr$, that is $J^2|_{\distr} = -I_{\distr}$. In this case the Riemannian volume, denoted $\vol_{g}$, coincides with the canonical Popp volume of the sub-Riemannian structure $(M,\distr,g)$, see \cite{BR-Popp}.

There always exists a canonical metric and linear connection, with non-vanishing torsion $\Tor$, called Tanno's connection $\nabla$. We denote by $ \mathrm{R}$ and $ \mathrm{Ric}$ the corresponding Riemann and Ricci tensor. The structure is \emph{Sasakian} if the following tensors vanish:
\begin{equation}
 Q(X,Y)= (\nabla_Y J)X=0,\qquad \tau(X) = \Tor(X_0,X) = 0, \qquad \forall X,Y \in TM.
\end{equation}

\subsubsection{Young diagram and curvature}

A horizontal curve $\gamma$ is a geodesic if and only if there exists a constant $h_{0} \in \R$ such that (cf.\ \cite[Lemma~6.7]{ABR-contact})
\begin{equation}\label{eq:hzero}
\nabla_{\ganf}\ganf=h_{0}J\ganf.
\end{equation} 
All non-trivial geodesic have the same Young diagram, with two columns and $3$ superboxes. We label them according to the convention of Figure \ref{f:YdMCP} in the Introduction, with $n=2d+1$ and $k=2d$.

\begin{lemma}\label{l:sasaky}
Let $M$ be a $(2d+1)$-dimensional contact Sasakian manifold and let $\mis=e^{-\psi}\vol_{g}$ be a smooth measure. Then along a geodesic $\gamma$ we have
\begin{align}
\mathfrak{Ric}^{a}_{\gamma}(t) & =  0, \\
\mathfrak{Ric}^{b}_{\gamma}(t)  & = \frac{1}{\|\dot\gamma\|^2}\mathrm{R}(\ganf,J\ganf,J\ganf,\ganf)  + h_0^2, \\
\mathfrak{Ric}^{c}_{\gamma}(t) & =  \frac{1}{\|\dot\gamma\|^2}\left( \mathrm{Ric}(\ganf) - \mathrm{R}(\ganf,J\ganf,J\ganf,\ganf)\right)+ \frac{1}{4}h_0^2(2d-2),
\end{align}
where in the right hand side we omitted the explicit dependence on $t$. Moreover
\begin{equation}\label{eq:rhom}
\rho_{\mis,\g}=-g(\nabla \psi, \ganf),\qquad \dot \rho_{\mis,\g}=-\nabla^{2}\psi(\ganf,\ganf)-h_{0}g(\nabla \psi,J\ganf).
\end{equation}
 \end{lemma}
 \begin{proof} The sub-Riemannian Ricci curvatures are computed in \cite[Thm.\ 6.3]{ABR-contact}. Moreover, from the formulas of \cite[p.~402]{ABR-contact}, it follows that $\rho_{\vol_{g},\g}=0$ for any geodesic.  Since $\mis=e^{-\psi}\vol_{g}$, we have that
 $\rho_{\mis,\g}=-g(\nabla \psi, \ganf)$. It follows that
\begin{equation}
 \dot \rho_{\mis,\g}=-\nabla_{\ganf} \left(g(\nabla \psi, \ganf)\right)= -\nabla^{2}\psi(\ganf,\ganf)-g(\nabla \psi,\nabla_{\ganf}\ganf).
\end{equation}
We conclude using equation \eqref{eq:hzero}.
 \end{proof}
 \begin{rmk}
Structures admitting a measure $\mis_{0}$ such that $\rho_{\mis_{0},\g}=0$ along every geodesic are called \emph{unimodular} in \cite{ABP}.
\end{rmk}

To compute the sub-Riemannian Bakry-Émery Ricci curvature we use Remark \ref{rmk:howtocomputesrBE}. In the Sasakian case, the three superboxes are denoted $a,b,c$, and their sizes are $1$, $2$ and $2d-2$, respectively. Therefore using Lemma~\ref{l:sasaky} we obtain
\begin{align}
 \mathfrak{Ric}^{N,a}_{\mis,\gamma}(t)&= 0, \label{eq:besas-a}\\
 \mathfrak{Ric}^{N,b}_{\mis,\gamma}(t)&= \frac{1}{\|\dot\gamma\|^2} \mathrm{R}(\ganf,J\ganf,J\ganf,\ganf)  + h_0^2 +\frac{1}{2d}\left(\nabla^{2}\psi(\ganf,\ganf) +h_{0}g(\nabla \psi,J\ganf) \right) \label{eq:besas-b}\\
 & \qquad -\frac{2d+1}{N-2d-1}\frac{g(\nabla \psi,\ganf)^{2}}{(2d)^2}, \nonumber \\
 \mathfrak{Ric}^{N,c}_{\mis,\gamma}(t)&= \frac{1}{\|\dot\gamma\|^2}\left(\mathrm{Ric}(\ganf) - \mathrm{R}(\ganf,J\ganf,J\ganf,\ganf)\right) + \frac{1}{4}h_0^2(2d-2)  \label{eq:besas-c}\\
 & \qquad +\frac{2d-2}{2d}\left(\nabla^{2}\psi(\ganf,\ganf)+h_{0}g(\nabla \psi,J\ganf)\right) \nonumber \\  & \qquad -\frac{(2d+1)(2d-2)}{N-2d-1}\frac{g(\nabla \psi,\ganf)^{2}}{(2d)^2}.  \nonumber
 \end{align}
Specifying Theorem~\ref{t:srbericci-sharp} to this setting, and using the model space coefficients of Section~\ref{ex:2}, we obtain the following statement.
\begin{theorem}\label{t:accadued}
Let $M$ be a $(2d+1)$-dimensional contact Sasakian manifold and $\mis=e^{-\psi}\vol_{g}$ be a smooth measure. Let $\gamma :[0,1] \to M$ be a minimizing geodesic between $x$ and $y$, with $(x,y)\notin \cut(M)$. Assume that there exists $N>2d+1$ and $\k_b,\k_c\in \R$ such that
\begin{align} \label{eq:duericci}
 \mathfrak{Ric}^{N,b}_{\mis,\gamma}(t)\geq \frac{N-1}{2d}\k_{b},\qquad  \frac{1}{2d-2}\mathfrak{Ric}^{N,c}_{\mis,\gamma}(t)\geq \frac{N-1}{2d}\k_{c},
\end{align}
with the convention that, if $d=1$, the second assumption can be omitted. Then
\begin{equation}\label{eq:137}
\beta_{t}(x,y)^{\frac{1}{N-1}}\geq t^{\frac{1}{N-1}}\left(\frac{\sin(t\alpha)}{\sin (\alpha)}\right)^{\frac{2d-2}{2d}} \left(\frac{\sin(t\theta)}{\sin (\theta)} \frac{t\theta \cos(t\theta)-\sin(t\theta)}{\theta \cos(\theta)-\sin(\theta)}\right)^{\frac{1}{2d}},
\end{equation}
for all $t\in [0,1]$, where $\alpha=\sqrt{\k_{c}}$ and $\theta = \sqrt{\k_{b}}/2$.
\end{theorem}
The right hand side of \eqref{eq:137} is understood as an analytic function of $\alpha,\theta$, as explained in Section~\ref{ex:2}. If $\psi$ is constant, then we can set formally $N=2d+1$ in Theorem \ref{t:accadued}, and the Bakry-Émery Ricci curvature is given by the simple formulas in Lemma \ref{l:sasaky}. In this latter case we recover the results of \cite{LLZ-Sasakian}.

\subsection{Weighted Heisenberg group}\label{s:Heisenberg}

Let us consider the three-dimensional Heisenberg group $\mathbb{H}=\mathbb{H}^3$, that is $\R^3$ endowed with the  sub-Riemannian structure defined by the global orthonormal frame
\begin{equation}
X_{1}=\frac{\partial}{\partial x_1}-\frac{x_{2}}{2}\frac{\partial}{\partial x_3},\qquad X_{2}=\frac{\partial}{\partial x_2}+\frac{x_{1}}{2}\frac{\partial}{\partial x_3}.
\end{equation}
It is well-known that this structure is Sasakian, with the canonical choice of contact form $\omega = dx_3 - \tfrac{1}{2}(x_1 dx_2  - x_2dx_1)$. Furthermore, $\vol_{g}$ is proportional to the Lebesgue measure of $\mathbb{R}^3$. We equip $\mathbb{H}$ with the weighted measure $\mis=e^{-\psi}\vol_{g}$, and we follow the notation of Section~\ref{s:61}.

Out goal is to apply Theorem~\ref{t:accadued} to $\mathbb{H}$, for which $\mathrm{R}(X,Y,Y,X)=0$ for all horizontal $X,Y$. Since $d=1$, we only need to provide a lower bound of the form:
\begin{equation}\label{eq:zetab}
 \mathfrak{Ric}^{N,b}_{\mis,\gamma}(t)=h_0^2 +\mathfrak{Z}(t)\geq \frac{N-1}{2d}\k_{b}.
 \end{equation}
For simplicity, we restrict to the case $\kappa_{b}=0$, which yields simpler polynomial bounds on the distortion coefficient. Recall from the previous section that
\begin{align}  \label{eq:zetaa}
\mathfrak{Z}(t)
&=\frac{1}{2}\left(\nabla^{2}\psi(\ganf,\ganf) +h_{0}g(\nabla \psi,J\ganf) \right) -\frac{3}{N-3}\frac{g(\nabla \psi,\ganf)^{2}}{4}.
\end{align}
We remark that, in this case, $\nabla$ coincides with the Tanaka-Webster connection of the contact structure. Let us denote the the horizontal gradient and the symmetrized horizontal Hessian of a smooth function $\psi$ by
\begin{equation}
\nabla_H \psi=(X_{1}\psi) X_{1} +(X_{2}\psi) X_{2}, \qquad (D_H^{2}\psi)^{*} = \left(\frac{X_i X_j\psi + X_j X_i\psi}{2}\right)_{i,j=1,2}.
\end{equation}
Let $\mathcal{B}_R(0)$ be the metric ball of radius $R>0$ centered at the origin, and set
\begin{equation}
L_R:= \sup_{\mathcal{B}_R(0)} \|\nabla_H \psi\|, \qquad C_R:=\inf_{\mathcal{B}_R(0)}(D^2_H\psi)^*,
\end{equation}
where the infimum denotes the infimum of the eigenvalues of the quadratic forms $(D_H^{2}\psi)^{*}(z)$ for $z \in \mathcal{B}_R(0)$. From \eqref{eq:zetab} and \eqref{eq:zetaa} we deduce the following lower bound.

\begin{lemma} \label{l:exempio}
Let $(x,y)\notin \cut(\mathbb{H})$, and let $\gamma:[0,1]\to \mathbb{H}$ be the geodesic joining $x$ with $y$. Let $R>0$ such that $x,y\in \mathcal{B}_R(0)$. Then it holds
\begin{equation}\label{eq:exempio}
\mathfrak{Ric}^{N,b}_{\mis,\gamma}(t) \geq  \left[ \frac{1}{2}\left(C_R - \frac{|h_0|}{\|\dot\gamma\|} L_R\right) - \frac{3}{4(N-3)}L_R^2 \right]\|\dot\gamma\|^2.
\end{equation}
\end{lemma}
From Lemma~\ref{l:exempio} and Theorem~\ref{t:accadued} we obtain the following result.
\begin{cor}\label{c:convex}
Let $\mathbb{H}$ be the three-dimensional Heisenberg group, equipped with the smooth measure $\mis=e^{-\psi}\vol_{g}$. Assume that for some $R>0$, it holds
\begin{equation}\label{eq:hconv}
C_R=\inf_{\mathcal{B}_R(0)}(D^2_H\psi)^*>0.
\end{equation}
Let $(x,y)\notin \cut(\mathbb{H})$, with $x,y \in \mathcal{B}_R(0)$. Assume that the unique geodesic $\gamma:[0,1]\to \mathbb{H}$ joining $x$ with $y$ is such that
\begin{equation}\label{eq:farfalla}
|h_0| < \frac{C_R}{L_R} d_{SR}(x,y).
\end{equation}
Then there exists $N_0>5$ such that
\begin{equation}\label{eq:convexestimate}
\beta_t(x,y) \geq t^{N_0},\qquad \forall t\in [0,1].
\end{equation}
\end{cor}
\begin{proof}
Under condition \eqref{eq:farfalla}, and since $d_{SR}(x,y) =\|\dot\gamma\|$, we have that the first term in the lower bound \eqref{eq:exempio} is strictly positive. We can then ensure, by choosing $N$ sufficiently large, that the whole right hand side of \eqref{eq:exempio} is non-negative. In particular this means that if $N\geq N_0$ with
\begin{equation}
N_0 = 5 +\frac{3}{2}\frac{L_R^2}{C_R - |h_0|L_R/\|\dot\gamma\|},
\end{equation}
then $\mathfrak{Ric}^{N,b}_{\mis,\gamma}(t)  \geq 0$. We then conclude easily by Theorem \ref{t:accadued}.
\end{proof}
We provide some comments on Corollary \ref{c:convex}.
\begin{rmk}
Notice that \eqref{eq:convexestimate} is coherent with the well-known fact that, if $\beta_t(x,y)\geq t^\alpha$ for some $\alpha>0$, then $\alpha$ is greater than the geodesic dimension of the Heisenberg group, that is $\alpha \geq 5$, see \cite[Sec.\ 8.2]{BR-G1}.
\end{rmk}
\begin{rmk}
 We recall that the parameter $h_0$ controls the spiraling of the geodesic. Furthermore, it is well-known that if $(x,y)\notin \cut(\mathbb{H})$, then for the corresponding geodesic we have $|h_0|< 2\pi$.
\end{rmk}
\begin{rmk} 
Any smooth function which is $C$-convex with respect to the Euclidean metric satisfies \eqref{eq:hconv}. The same holds for any horizontally $C$-convex function in the sense of \cite{DGN-Convex,Balogh-Convex}. In this case one can choose $C=C_R$ in \eqref{eq:hconv} for all $R>0$. The value of $N_0$ will still depend on $x,y$ through $d_{SR}(x,y)$, $h_0$ and $R$. For example, one can choose 
$\psi : \mathbb{H} \to \R$
\begin{equation}
\psi(x)=\frac{1}{2}(x_{1}^{2}+x_{2}^{2}),
\end{equation}
which is horizontally $1$-convex and with cylindrical symmetry. In this case we can choose $C_R=1$ for any $R>0$ in Corollary \ref{c:convex}. Furthermore
\begin{equation}
\|\nabla_H \psi\|^{2}=x_{1}^{2}+x_{2}^{2},\qquad L_R=\sup_{\mathcal{B}_R(0)} \|\nabla_H \psi\|= R.
\end{equation}
When either $x$ or $y$ go to infinity (i.e.\ $R\to +\infty$) then the corresponding $N_0$ is not bounded.   
This is the sub-Riemannian analogue of the well-known fact that $\R^{n}$ endowed with the Gaussian measure does not satisfy any global $\mathrm{CD}(0,N)$ condition for finite $N$.
\end{rmk}
\begin{rmk}
In the spirit of Theorem \ref{t:srbe} one might require separately a lower bound on curvature (which in this example is zero), and an upper bound on the geodesic volume derivative. Let then $\gamma$ be a geodesic joining $x$ with $y$, and assume that $\gamma$ is contained $\mathcal{B}_R(0)$. We have in this case
\begin{equation}\label{eq:ruspa}
\rho_{\mis,\gamma}(t) \leq d_{SR}(x,y)\sup_{\mathcal{B}_R(0)}\|\nabla_H \psi\| = d_{SR}(x,y) L_R.
\end{equation}
By Remark \ref{r:cgeqleq} we obtain that for any $x,y\in\mathcal{B}_R(0)$ with $(x,y)\notin \cut(\mathbb{H})$ we have
\begin{equation}\label{eq:stimaruspa}
\beta_t(x,y) \geq t^5 e^{d_{SR}(x,y) L_R  (t-1) }, \qquad \forall t\in [0,1].
\end{equation}
Estimate \eqref{eq:stimaruspa} can be stronger or weaker than the one provided by Corollary~\ref{c:convex}, depending on the values of the parameters. This is similar to what happens already in the weighted Euclidean case.
\end{rmk}

\subsection{3-Sasakian manifolds}

We use the notation and conventions of \cite[Sec.\ 5]{RS17}, to which we refer to for more details. A \emph{$3$-Sasakian structure} on a smooth manifold $M$ of dimension $4d+3$, with $d\geq 1$, is a collection $\{\phi_\alpha,\eta_\alpha,\xi_\alpha,g\}_\alpha$, with $\alpha=I,J,K$, of three contact metric structures, where $g$ is a Riemannian metric, $\eta_\alpha$ is a one-form, $\xi_\alpha$ is the Reeb vector field and $\phi_\alpha : \Gamma(TM) \to \Gamma(TM)$ satisfy
\begin{equation}
2g(X,\phi_\alpha Y) = d\eta(X,Y), \qquad \qquad \forall X,Y \in TM.
\end{equation}
The three structures are Sasakian, and $\phi_I,\phi_J,\phi_K$ satisfy quaternionic-like compatibility relations. A natural sub-Riemannian structure is given by the restriction of the Riemannian metric $g$ to the distribution
\begin{equation}
\distr = \bigcap_{\alpha=I,J,K} \ker \eta_\alpha.
\end{equation}
The three Reeb vector fields $\xi_\alpha$ are an orthonormal triple, orthogonal to $\distr$. We denote by $\vol_{g}$ the corresponding Riemannian measure, which is proportional to the canonical Popp measure of th sub-Riemannian structure. For $3$-Sasakian structures we adopt as a reference connection the Levi-Civita connection $\nabla$ of $g$.
 
\subsubsection{Young diagram and curvature}

A horizontal curve $\gamma$ is a geodesic if and only if there exists three constants $v_I,v_J,v_K \in \R$ such that (cf.\ \cite[Lemma 37]{RS17})
\begin{equation}\label{eq:geod3sas}
\nabla_{\ganf}\ganf=\sum_{\alpha=I,J,K}v_{\alpha}\phi_{\alpha}\ganf.
\end{equation} 
In the following, we let $\|v\|^{2}:=\sum_{\alpha=I,J,K}v_{\alpha}^{2}$.

Any non-trivial geodesic has the same Young diagram, with two columns and $3$ superboxes. We label them according to the convention of Figure \ref{f:YdMCP} in the Introduction, with $n=4d+3$ and $k=4d$, and we label accordingly the corresponding Ricci curvatures. We are now ready to prove Theorem \ref{t:3Sas}, stated in the Introduction.

\subsubsection{Proof of Theorem \ref{t:3Sas}}

We will apply Theorem \ref{t:mcp}. Thanks to the homogeneity property of the sub-Riemannian Ricci curvature (see Appendix \ref{s:prel}), it is sufficient to check the assumptions for unit-speed geodesics. The sub-Riemannian Ricci curvatures of a $3$-Sasakian structure have been computed in \cite[Thm.\ 8]{RS17}. In particular for every unit-speed geodesic $\gamma$ it holds $\rho_{\mis,\g} =0$ and
\begin{align}
\mathfrak{Ric}^a_{\gamma}(t) & =  3\left(\tfrac{3}{4}\varrho(v)-\tfrac{7}{2}\|v\|^2-\tfrac{15}{8}\|v\|^4\right),\\
\mathfrak{Ric}^b_{\gamma}(t)& = 3(4 +5\|v\|^2), \\
\mathfrak{Ric}^c_{\gamma}(t) &  =(4d-4)(1+ \|v\|^2),
\end{align}
where $v = (v_I,v_J,v_K)$ are as in \eqref{eq:geod3sas}. In the above formulas $\varrho(v)$ is a sectional-like curvature invariant, given by
\begin{equation}
\varrho(v) := \sum_{\alpha=I,J,K} R_g(\dot\gamma,Z_\alpha,Z_\alpha,\dot\gamma),
\end{equation}
where $R_g$ is the (Levi-Civita) Riemannian curvature of the $3$-Sasakian structure, and the vectors $Z_I,Z_J,Z_K \in \distr$ are
\begin{equation*}
Z_I := (v_J \phi_K  - v_K \phi_J) \dot\gamma, \qquad Z_J := (v_K \phi_I  - v_I \phi_K) \dot\gamma, \qquad Z_K := (v_I \phi_J  - v_J \phi_I) \dot\gamma.
\end{equation*}
Assume now \eqref{eq:boundrhoa}, thus $\varrho(v) \geq \sum_{\alpha} K \|Z_\alpha\|^2 = 2 K \|v\|^2$. Thus
\begin{align}
\tfrac{1}{3}\mathfrak{Ric}^a_{\gamma(t)} & \geq  \k_a(v) := \|v\|^2\left(\tfrac{3}{2}K-\tfrac{7}{2}-\tfrac{15}{8}\|v\|^2\right),\\
\tfrac{1}{3}\mathfrak{Ric}^b_{\gamma(t)} & \geq  \k_b(v) := ( 4+5\|v\|^2),\\
\mathfrak{Ric}^c_{\gamma(t)} & \geq  \k_c(v) := 0.
\end{align}
for any unit-speed geodesic $\gamma$. Conditions \eqref{eq:conditions2} of Theorem \ref{t:mcp} are equivalent to
\begin{equation}
	\frac{35}{2}\|v\|^4+(26+6K)\|v\|^2+16 \geq 0,
\end{equation}
which is verified independently on $v$ provided that $K\geq -\tfrac{1}{3}(13+ 2\sqrt{70}) \simeq -9.91$. \hfill \qedsymbol

\appendix

\section{Sub-Riemannian curvature and canonical moving frames} \label{s:prel}

We assume the reader to be familiar with the basic definitions of sub-Riemannian geometry. We refer to \cite[Sec.\ 2]{BR-G1} for a minimal background, and to \cite{nostrolibro} for a comprehensive introduction. The material presented in this appendix has been developed in \cite{curvature,BR-comparison,BR-connection}, following the pioneering works of Agrachev-Zelenko \cite{agzel1,agzel2} and Zelenko-Li \cite{lizel}.

\subsection{Notation} In what follows $M$ is a smooth, connected $n$-dimensional manifold (where $n\geq 3$), equipped with a bracket-generating distribution $\distr$ of rank $k$. The distribution is endowed with an inner product, defining the sub-Riemannian distance $d_{SR}$. The Hamiltonian associated with the sub-Riemannian structure is denoted by $H$, and $\vec{H}$ denotes the corresponding Hamiltonian vector field.

A geodesic is a horizontal curve parametrized with constant speed whose short arcs realize the sub-Riemannian distance. A geodesic is normal if there exists a lift $\lambda:[0,T]\to T^{*}M$ such that $\lambda(t)=e^{t\vec H}(\lam_0)$ for some $\lambda_0\in T^{*}_{x}M$ and $\gamma(t) =\pi(\lambda(t))$. The lift $\lambda$ is called normal extremal.

 Recall that $\cut(x)$ is the complement of the set of points where $d^2_{SR}(x,\cdot)$ is smooth. It is closed and nowhere dense in $M$ \cite{Agrasmoothness,RT-MorseSard}. Let $\cut(M) = \{(x,y)\mid x\in M,\; y\in \cut(x)\}$ (see also \cite[Def.\ 18]{BR-G1}).
 
\subsection{Ample and equiregular curves}\label{s:gfyd}

Let $\gamma$ be a smooth horizontal curve, and consider a smooth admissible extension of the tangent vector, namely a horizontal vector field $\tanf$ such that $\tanf|_{\gamma(t)} = \dot{\gamma}(t)$.
The \emph{flag} of $\gamma$ is the sequence of subspaces (cf.~Definition~\ref{d:flag})
\begin{equation}
\DD_{\gamma(t)}^i :=  \spn\{\mc{L}_\tanf^j (X)|_{\gamma(t)} \mid  X \in \Gamma(\distr),\, j \leq i-1\} \subseteq T_{\gamma(t)} M, \qquad \forall\, i \geq 1,
\end{equation}
where $\mc{L}_{\tanf}$ denotes the Lie derivative in the direction of $\tanf$. The definition is well posed, namely it does not depend on the choice of the admissible extension $\tanf$ (see \cite[Sec. 3.4]{curvature}). Observe that $\DD_{\gamma(t)}^i \subseteq \DD_{\gamma(t)}^{i+1}$ for all $i \geq 1$, and $\DD_{\gamma(t)}^1 = \distr_{\gamma(t)}$.

The \emph{growth vector} of $\gamma$ is the sequence of integer numbers 
\begin{equation}
\mathcal{G}_{\gamma(t)} := \{\dim \DD_{\gamma(t)}^1,\dim \DD_{\gamma(t)}^2,\ldots\}.
\end{equation}
We say that the smooth horizontal curve $\gamma$ is
\begin{itemize}
\item[(a)] \emph{equiregular} if $\dim \DD_{\gamma(t)}^i$ does not depend on $t$ for all $i \geq 1$,
\item[(b)] \emph{ample} if for all $t$ there exists $m \geq 1$ such that $\dim \DD_{\gamma(t)}^{m} = \dim T_{\gamma(t)}M$.
\end{itemize}
Assume from now on that $\gamma$ is ample and equiregular. The smallest integer $m \geq 1$ such that $\dim \DD_{\gamma(t)}^{m} = \dim T_{\gamma(t)}M$ is called \emph{step} of $\gamma$. Let
\begin{equation}
d_i =  \dim \DD_{\gamma(t)}^i -\dim \DD_{\gamma(t)}^{i-1}, \qquad i \geq 1,
\end{equation}
with the convention that $\dim\DD_{\gamma(t)}^0 =0$. It is easy to show that $d_1 \geq d_2 \geq \ldots \geq d_m$, cf.\ \cite[Lemma 3.5]{curvature}. 

\subsection{Young diagrams}\label{s:young}
To any ample and equiregular curve, we associate a Young tableau $\y$, with $m$ columns of length $d_{i}$, for $i=1,\ldots,m$, as follows:
\begin{figure}[!ht]
\centering
\includegraphics[scale=1]{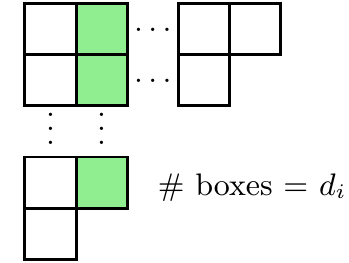}
\end{figure}

The total number of boxes in $\y$ is equal to the dimension of the manifold $\sum_{i=1}^m d_i = n$. The diagram $\y$ is a way to encode the data of the growth vector of $\gamma$.

Let $n_1,\ldots,n_k$ be the lengths of the rows, where $k = \rank \distr$. We employ the notation $ai \in \y$ to denote the generic box of the diagram, where $a=1,\ldots,k$ is the row index, and $i=1,\ldots,n_a$ is the progressive box number, starting from the left, in the specified row.

\begin{figure}[!ht]
\centering
\includegraphics[scale=1]{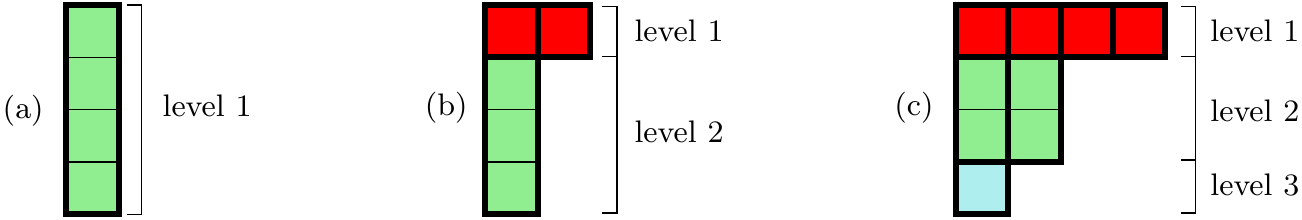}
\caption{Levels (shaded regions) and superboxes (delimited by bold lines) for some Young diagrams}\label{f:Yd2}
\end{figure}

We collect rows with the same length in $\y$, and we call them \emph{levels}.  If a level is the union of $r$ rows $\y_1,\ldots,\y_r$, then $r$ is called the \emph{size} of the level. The set of all the boxes $ai \in\y$ that belong to the same column and the same level of $\y$ is called \emph{superbox}. Notice that that two boxes $ai$, $bj$ are in the same superbox if and only if $ai$ and $bj$ are in the same column of $\y$ and in possibly distinct row but with same length, i.e., if and only if $i=j$ and $n_a = n_b$ (see Fig.~\ref{f:Yd2}). The Greek letters $\alpha,\beta$ are used to denote the generic superbox of the Young diagram. Sometimes, with an abuse of notation that should not cause confusion, we use $\lev$ to denote the generic level of the Young diagram, and if $\ell = \ell_\lev$ is its length, the superboxes belonging to that level are denoted $\lev_1,\dots,\lev_\ell$.

\subsection{Normal form matrices}\label{s:matrices}
Given a Young diagram $\y$, we define the two associated $n\times n$ \emph{matrices} $\Gamma_1=\Gamma_1(\y)$ and $\Gamma_2=\Gamma_2(\y)$ as follows.
 For $a,b = 1,\dots,k$, $i=1,\dots,n_a$, $j=1,\dots,n_b$:
\begin{equation}
[\Gamma_1]_{ai,bj} := \delta_{ab}\delta_{i,j-1}, \qquad
[\Gamma_2]_{ai,bj} := \delta_{ab}\delta_{i1}\delta_{j1}.
\end{equation}
It is convenient to regard $\Gamma_1$ and $\Gamma_2$ as block diagonal matrices:
\begin{equation}
\Gamma_i(\y) := \begin{pmatrix} 
\Gamma_i(\y_1) &  &    \\
 &  \ddots & \\
 &   & \Gamma_i(\y_k)
\end{pmatrix}, \qquad i =1,2,
\end{equation}
where $\y_a$, for $a=1,\dots,k$ denotes the $a$-th row of $\y$. Thus the $a$-th block in the above formula corresponds to the $n_a\times n_a$ matrices
\begin{equation}\label{eq:Gamma}
\Gamma_1(\y_a) := \begin{pmatrix}
0 & \mathbbold{1}_{n_a-1} \\
0 & 0
\end{pmatrix} , 
\qquad \Gamma_2(\y_a) := \begin{pmatrix}
1 & 0 \\
0 & \mathbbold{0}_{n_a-1}
\end{pmatrix},
\end{equation}
where $\mathbbold{1}_{m}$ is the $m \times m$ identity matrix and $\mathbbold{0}_{m}$ is the $m \times m$ zero matrix. Notice that the matrices $A=\Gamma_{1}^*(\y)$ and $B=\Gamma_{2}(\y)$ satisfy the Kalman rank condition
\begin{equation}\label{eq:Kalman}
\rank\{B, AB, \dots, A^{n-1} B\}=n.
\end{equation}
Analogously, the matrices $A_a=\Gamma_{1}^*(\y_{a})$, $B_a = \Gamma_2(\y_a)$ satisfy \eqref{eq:Kalman} with $n=n_{a}$.

\subsection{Sub-Riemannian Jacobi fields}

Let $\lambda(t)= e^{t\vec{H}}(\lambda_0)$, $t \in [0,1]$ be an integral curve of the Hamiltonian flow. For any smooth vector field $\xi(t)$ along $\lambda(t)$, the dot denotes the Lie derivative in the direction of $\vec{H}$, namely
\begin{equation}
\dot{\xi}(t) := \left.\frac{d}{d\eps}\right|_{\eps=0} e^{-\eps \vec{H}}_* \xi(t+\eps).
\end{equation}
A vector field $\J(t)$ along $\lambda(t)$ is a \emph{Jacobi field} if it satisfies the equation
\begin{equation}\label{eq:defJF}
\dot{\J} = 0.
\end{equation}
Jacobi fields along $\lambda(t)$ are of the form $\J(t) = e^{t\vec{H}}_* \J(0)$, for some unique initial condition $\J(0) \in T_{\lambda_0} (T^*M)$, and the space of solutions of \eqref{eq:defJF} is a $2n$-dimensional vector space. We define the smooth sub-bundle $\ver\subset T^*M$ with Lagrangian fibers:
\begin{equation}
\ver_{\lambda} := \ker \pi_*|_{\lambda} = T_\lambda(T^*_{\pi(\lambda)} M) \subset T_{\lambda}(T^*M), \qquad \lambda \in T^*M,
\end{equation}
which we call \emph{vertical sub-bundle}.

Let $\gamma :[0,1] \to M$ be a normal geodesic, projection of $\lambda(t) = e^{t\vec{H}}(\lambda_0)$, for some $\lambda_0 \in T^*M$. Consider the family of $n$-dimensional subspaces generated by a set of independent Jacobi fields $\J_1(t),\dots,\J_n(t)$ along $\lambda(t)$, that is
\begin{equation}
\mathcal{L}_t=\spn\{\J_1(t),\dots,\J_n(t)\} \subset T_{\lambda(t)}(T^*M).
\end{equation}
Since $\mathcal{L}_t = e^{t\vec{H}}_* \mathcal{L}_0$, then $\mathcal{L}_t$ is Lagrangian if and only if it is Lagrangian at time $t=0$.

Let $\sigma$ be the symplectic structure of $T^*M$. Fix a Darboux moving frame along $\lambda(t)$, that is smooth vector fields $E_i(t),F_j(t) \in T_{\lambda(t)} (T^*M)$, $i,j=1,\dots,n$, such that
\begin{equation}
\sigma(E_i,F_j) - \delta_{ij} = \sigma(E_i,E_j) = \sigma(F_i,F_j) = 0, \qquad \forall i,j=1,\ldots,n,
\end{equation}
and such that $E_1(t),\ldots,E_n(t)$ generate the vertical subspace $\mathcal{V}_{\lambda(t)}=\ker \pi_{*}|_{\lambda(t)}$:
\begin{equation}
\mathcal{V}_{\lambda(t)}= \spn\{E_1(t),\dots,E_n(t)\}, \qquad \forall t \in [0,1].
\end{equation}
We denote with $X_i(t):=\pi_* F_i(t)$, for $i=1,\dots,n$, the corresponding moving frame along the normal geodesic $\gamma(t) = \pi(\lambda(t))$, $t\in [0,1]$.
\begin{definitionappendix}
We say that $\{E_i(t),F_i(t)\}_{i=1}^{n}$ is a \emph{moving Darboux frame} along the extremal $\lambda(t)$, and that $X_1(t),\dots,X_n(t)$ is the corresponding \emph{moving frame} along the geodesic $\gamma(t)$.
\end{definitionappendix}

We identify $\mathcal{L}_t = \spn\{\J_1(t),\dots,\J_n(t)\}$ with a smooth family of $2n \times n$ matrices
\begin{equation}\label{eq:jacmatrix}
\mathbf{J}(t) = \begin{pmatrix}
M(t) \\
N(t)
\end{pmatrix}, \qquad t \in [0,1],
\end{equation}
such that, with respect to the given Darboux frame, we have
\begin{equation}\label{eq:howtowriteJ}
\J_{i}(t) = \sum_{j=1}^n E_{j}(t)  M_{ji}(t) +  F_{j}(t)N_{ji}(t), \qquad \forall i=1,\dots,n.
\end{equation}
We call $\mathbf{J}(t)$ a \emph{Jacobi matrix}, while the $n\times n$ matrices $M(t)$ and $N(t)$ represent respectively its ``vertical'' and ``horizontal'' components with respect to the decomposition induced by the Darboux moving frame
\begin{equation}
T_{\lambda(t)}(T^*M) = \mathcal{H}_{\lambda(t)} \oplus \mathcal{V}_{\lambda(t)}, \qquad \text{with} \qquad \mathcal{H}_{\lambda(t)} := \spn\{F_1(t),\dots,F_n(t)\}.
\end{equation}

Jacobi matrices are solutions of a general Hamiltonian system or, equivalently, a Riccati-type matrix equation. The precise statement is as follows. Its proof follows directly form the properties of $H$, see for example \cite[Lemma 24]{BR-G1}.

\begin{lemmaappendix}\label{l:subspacesriccati}
For any Darboux frame along $\lambda(t)$ there exist smooth families of $n\times n$ matrices $A(t),B(t),R(t)$, $t \in[0,1]$, with $B(t),R(t)$ symmetric and $B(t) \geq 0$, such that any Jacobi matrix $\mathbf{J}(t)$ is a solution of
\begin{equation}\label{eq:mainequation}
\frac{d}{dt} \begin{pmatrix}
M \\
N
\end{pmatrix} = \begin{pmatrix}
-A(t)^* & - R(t) \\
B(t) & A(t)
\end{pmatrix} \begin{pmatrix}
M \\
N
\end{pmatrix}.
\end{equation}
On any interval $I \subseteq [0,1]$ such that $N(t)$ is non-degenerate, the matrix $V(t):= M(t)N(t)^{-1}$ satisfies the Riccati equation
\begin{equation}
\dot{V} + A(t)^* V + VA(t) + V B(t) V + R(t) = \mathbbold{0}, \qquad
\end{equation}
The associated family of subspaces $\mathcal{L}_t$ is Lagrangian if and only if $V(t)$ is symmetric.
\end{lemmaappendix}

\subsection{Canonical Darboux frame}\label{a:canframe}

There exists a canonical choice of moving Darboux frames, in terms of which the Hamiltonian system of Lemma \ref{l:subspacesriccati} takes a simple normal form. This frame is uniquely defined up to constant orthogonal transformations that, roughly speaking, respect the structure of the Young diagram. The following theorem is the main result of \cite{lizel}.
\begin{theoremappendix}\label{p:can} 
Let $\gamma(t)$ be an ample and equiregular geodesic with Young diagram $\y$, and let $\lambda(t)$ be its normal extremal lift. Then, there exists a moving Darboux frame $\{E_{ai}(t),F_{ai}(t)\}_{ai \in \y}$ along $\lambda(t)$ such that the Hamiltonian system of Lemma \ref{l:subspacesriccati} takes a normal form, with
\begin{equation}
A(t) = \Gamma_1^*(\y), \qquad B(t) =\Gamma_2(\y),
\end{equation}
are the constant matrices defined in Section \ref{s:matrices}, and the symmetric matrix $R(t)$ is \emph{normal} in the sense of Zelenko-Li (see Definition~\ref{d:normal}).

If $\{\wt{E}_{ai},\wt{F}_{ai}\}_{ai \in \y}$ is another moving Darboux frame verifying the above properties, for some normal matrix $\wt{R}(t)$, then for any superbox $\alpha$ of size $r$ there exists an orthogonal constant $r\times r$ matrix $O^\alpha$ such that
\begin{equation}
\wt{E}_{ai} = \sum_{bj \in \alpha} O^\alpha_{ai,bj} E_{bj}, \qquad \wt{F}_{ai} = \sum_{bj \in \alpha} O^\alpha_{ai,bj} F_{bj}, \qquad \forall ai \in \alpha.
\end{equation}
\end{theoremappendix}

\begin{definitionappendix}
The frame $\{E_{ai}(t),F_{ai}(t)\}_{ai \in \y}$ of Theorem \ref{p:can} is called \emph{canonical moving Darboux frame} along $\lambda(t)$. The frame $\{X_{ai}(t)\}_{ai \in \y}$, defined by $X_{ai}(t) = \pi_* F_{ai}(t)$ is the corresponding \emph{canonical moving frame} along $\gamma(t)$. 
\end{definitionappendix}
It is not hard to check that, in the Riemannian case, canonical moving frames along $\gamma(t)$ are precisely the parallel and orthonormal ones (see for instance~\cite{BR-comparison}).

\begin{definitionappendix}\label{d:normal}
A $n\times n$ matrix $R$, whose entries are labelled according to the entries of a Young diagram $\y$, is \emph{normal} in the sense of Zelenko-Li if it satisfies:
\begin{itemize}
\item[(i)] global symmetry: for all $ai,bj\in \y$
\[R_{ai,bj}=R_{bj,ai}.\]
\item[(ii)] partial skew-symmetry: for all $ai,bi\in \y$ with $n_{a}=n_{b}$ and $i<n_{a}$ 
\[R_{ai,b(i+1)}=- R_{bi,a(i+1)}.\]
\item[(iii)] vanishing conditions: the only possibly non vanishing entries $R_{ai,bj}$ satisfy
\begin{itemize}
\item[(iii.a)] $n_{a}=n_{b}$ and $|i-j|\leq 1$,
\item[(iii.b)] $n_{a}>n_{b}$ and $(i,j)$ belong to the last $2n_{b}$ elements of Table~\ref{t:normaltable}.
\begin{table}[htbp]
\begin{center}
\caption{Vanishing conditions.}
\begin{tabular}{|c||c|c|c|c|c|c|c|c|c|c|c|c|c|}
\hline
$i$ & $1$ & $1$ & $2$ & $\cdots$ & $\ell$ & $\ell$ & $\ell+1$ & $\cdots$ & $n_{b}$ & $n_{b}+1$ & $\cdots$ & $n_a-1$ & $n_a$ \\
\hline
$j$ & $1$ & $2$ & $2$ & $\cdots$ & $\ell$ & $\ell+1$ & $\ell+1$ & $\cdots$ & $n_b$ & $n_b$ & $\cdots$ & $n_b$ & $n_b$ \\
\hline
\end{tabular}
\label{t:normaltable}
\end{center}
\end{table}
\end{itemize}
\end{itemize}
The sequence is obtained as follows: starting from $(i,j)=(1,1)$ (the first boxes of the rows $a$ and $b$), each next even pair is obtained from the previous one by increasing $j$ by one (keeping $i$ fixed). Each next odd pair is obtained from the previous one by increasing $i$ by one (keeping $j$ fixed). This stops when $j$ reaches its maximum, that is $(i,j) = (n_b,n_b)$. Then, each next pair is obtained from the previous one by increasing $i$ by one (keeping $j$ fixed), up to $(i,j) = (n_a,n_b)$. The total number of pairs appearing in the table is $n_b+n_a-1$.
\end{definitionappendix}

\subsection{Canonical structure} \label{s:cancurv}
Theorem \ref{p:can} defines several canonical objects along $\gamma(t)$, including the sub-Riemannian curvature. Let then $\{X_{ai}(t)\}_{ai \in \y}$ be a canonical moving frame along the ample and equiregular geodesic $\gamma(t)$. Such a frame is defined up to constant orthogonal transformations that mix only the $X_{ai}$'s belonging to the same superbox of $\y$. Thus, the following definitions are well posed for all $t$.
\begin{definitionappendix}\label{d:split}
The \emph{canonical splitting} of $T_{\gamma(t)} M$ is
\begin{equation}
T_{\gamma(t)}M = \bigoplus_{\alpha}S_{\gamma(t)}^{\alpha}, \qquad S_{\gamma(t)}^{\alpha}:=\spn\{ X_{ai}(t)\mid \, ai \in \alpha\},
\end{equation}
where the sum is over the superboxes $\alpha$ of $\y$. The dimension of $S_{\gamma(t)}^{\alpha}$ is equal to the size $r$ of the superbox $\alpha$, that is the number of boxes contained in $\alpha$.
\end{definitionappendix}

\begin{definitionappendix}
The \emph{canonical scalar product} is the positive quadratic form $\langle \cdot |\cdot \rangle_{\gamma(t)} : T_{\gamma(t)}M\times T_{\gamma(t)}M\to \R$ such that $\{X_{ai}(t)\}_{ai\in\y}$ is an orthonormal frame for $\langle \cdot |\cdot \rangle_{\gamma(t)}$.
\end{definitionappendix}
It is not difficult to show that the subset $\{X_{a1}\}_{a1\in \y}$ is an orthonormal frame for the sub-Rieman\-nian metric $g$, and thus $\langle \cdot |\cdot\rangle_{\gamma(t)}$ coincides with $g$ on $\distr_{\gamma(t)}$.

\begin{definitionappendix}
Let $\Pi_{\gamma(t)}$ be the orthogonal projection on $\distr_{\gamma(t)}$ with respect to $\langle \cdot |\cdot \rangle_{\gamma(t)}$. We define a non-negative quadratic form $\mathfrak{B}_{\gamma}(t):T_{\gamma(t)}M\times T_{\gamma(t)}M \to \R$ as
\[
\mathfrak{B}_{\gamma}(t)(v,w)=g(\Pi_{\gamma(t)}v,\Pi_{\gamma(t)}w), \qquad \forall v,w \in T_{\gamma(t)} M.
\]
\end{definitionappendix}

\begin{rmkappendix}\label{rmk:howtocomputeB}
The representative matrix of $\mathfrak{B}_{\gamma}(t)$, in terms of the basis $\{X_{ai}\}_{ai \in \y}$, is the matrix $B=\Gamma_2(\y)$ of Section \ref{s:matrices}. In particular, for any superbox $\alpha$ we have
\begin{equation}
\tr\left(\mathfrak{B}_{\gamma}(t)|_{S_{\gamma(t)}^{\alpha}}\right) = \begin{cases}
\mathrm{size}(\alpha) & \text{$\alpha$ is the first superbox of its level}, \\
0 & \text{otherwise}.
\end{cases}
\end{equation}
\end{rmkappendix}

\begin{definitionappendix}\label{d:curvature}
The \emph{canonical curvature} is the quadratic form  $\Rcan_{\g}(t): T_{\gamma(t)} M \times T_{\gamma(t)} M\to \R$ whose representative matrix, in terms of the basis $\{X_{ai}\}_{ai \in \y}$, is $R_{ai,bj}(t)$. In other words for all $v\in T_{\gamma(t)}M$ we have
\begin{equation}\label{eq:SR-dircurv}
\Rcan_{\g}(t)(v,v) := \sum_{ai,bj\in\y} R_{ai,bj}(t) v_{ai}v_{bj}, \qquad v = \sum_{ai\in\y} v_{ai} X_{ai}(t) \in T_{\gamma(t)}M.
\end{equation}
For any pair of superboxes $\alpha,\beta$, we denote the restrictions of $\mathfrak{R}_{\g}(t)$ on the appropriate subspaces by:
\begin{equation}
\mathfrak{R}^{\alpha\beta}_{\g}(t) : S^{\alpha}_{\gamma(t)} \times S^{\beta}_{\gamma(t)} \to \R.
\end{equation}
Finally, for any superbox $\alpha$, the \emph{canonical Ricci curvature} is the partial trace:
\begin{equation}
\mathfrak{Ric}_{\g}^\alpha(t):= \sum_{ai \in \alpha} \Rcan_{\g}^{\alpha\alpha}(t)(X_{ai}(t),X_{ai}(t)).
\end{equation}
\end{definitionappendix}
 \begin{rmkappendix}\label{rmk:howtocomputesrBE}
Let $N>n$ and fix a smooth measure $\mis$ on $M$. The sub-Rieman\-nian Bakry-Émery Ricci curvature $\Rcan_{\mis,\g}^{N,\lev}(t)$ was defined in \eqref{eq:srbakryemery} as the partial trace, over the subspace $S_{\gamma(t)}^\alpha$ associated with the superbox $\alpha$, of
\begin{equation}
\mathfrak{R}_{\mis,\gamma}^{N}(t)= 
\mathfrak{R}_{\gamma}(t)-\left(\frac{\dot \rho_{\mis,\g}(t)}{k}+\frac{n}{N-n}\frac{ \rho_{\mis,\g}^{2}(t)}{k^{2}}\right)\mathfrak{B}_{\gamma}(t),
\end{equation}
where $\rho_{\mis,\g}(t)$ is the geodesic volume derivative of Definition \ref{def:geodvolder}. Taking into account Remark \ref{rmk:howtocomputeB}, and letting $r_\alpha$  be the size of the superblock $\alpha$, we have
\begin{equation}
\mathfrak{Ric}_{\mis,\g}^{N,\alpha}(t)=
\mathfrak{Ric}_{\g}^\alpha(t) -r_\alpha\left(\frac{\dot \rho_{\mis,\g}(t)}{k}+\frac{n}{N-n} \frac{ \rho_{\mis,\g}^{2}(t)}{k^{2}}\right),
\end{equation}
if $\alpha$ is the first superblock of its level, and $\mathfrak{Ric}_{\mis,\g}^{N,\alpha}(t) = \mathfrak{Ric}_{\g}^\alpha(t)$ otherwise.
\end{rmkappendix}

\subsection{Homogeneity properties}

For all $c>0$, let $H_c := H^{-1}(c/2)$ be the Hamiltonian level set. In particular $H_1$ is the unit cotangent bundle: the set of initial covectors associated with unit-speed geodesics. Since the Hamiltonian function is fiber-wise quadratic, we have the following property for any $c>0$
\begin{equation}\label{eq:commutation}
e^{t \vec{H}}(c \lambda) = c e^{c t\vec{H}}(\lambda),
\end{equation}
where, for $\lambda \in T^*M$, the notation $c \lambda$ denotes the fiber-wise multiplication by $c$. 
%
The sub-Riemannian curvatures enjoy the following homogeneity property, proved in \cite[Thm.\ 4.7]{BR-connection}. The analogous property of the geodesic volume derivative follows from analogous properties of the canonical frame \cite[Prop.\ 4.9]{BR-connection}.
\begin{theoremappendix}\label{t:homogR}
Let $\gamma :[0,T] \to M$ an ample and equiregular geodesic, with normal extremal $\lambda : [0,T] \to T^*M$, and Young diagram $\y$. Let $c>0$ and consider the reparametrization $\gamma^c:[0,T/c] \to M$ defined by $\gamma^c(t) = \gamma(ct)$, which is again ample and equiregular, with the same Young diagram. The corresponding normal extremal is $\lambda^c :[0,T/c] \to T^*M$ with $\lambda^c(t) = c \lambda(ct)$. For any superbox $\alpha \in \y$, let $|\alpha|$ denote the column index of $\alpha$. Then, we have
\begin{equation}
\mathfrak{R}^{\alpha\beta}_{\gamma^{c}}(t) = c^{|\alpha|+ |\beta|} \mathfrak{R}^{\alpha\beta}_{\gamma}(ct).
\end{equation}
In particular, for the Ricci curvatures it holds
\begin{equation}
\mathfrak{Ric}^{\alpha}_{\gamma^c}(t) = c^{2|\alpha|} \mathfrak{Ric}^\alpha_{\gamma}(ct).
\end{equation}
Furthermore, for the geodesic volume derivative, it holds
\begin{equation}
\rho_{\mis,\gamma^c}(t)= c\rho_{\mis,\gamma}(ct).
\end{equation}
\end{theoremappendix}
\begin{rmkappendix}
In the Riemannian setting, $\y$ has only one superbox with $|\al|=1$  (see Fig.~\ref{f:Yd2}). Then $\mathfrak{R}_{\gamma}(t):=\mathfrak{R}^{\alpha\alpha}_{\gamma}(t)$ is homogeneous of degree $2$ as a function of $\dot\gamma(t)$.
\end{rmkappendix}
\section{Matrix Riccati comparison} \label{s:riccati}

We consider the following non-autonomous matrix Riccati equation 
\begin{equation}
\dot{X} = \mathrm{R}(X;t):=M(t)_{11} +  X M(t)_{12} + M(t)_{12}^* X + X M(t)_{22} X, 
\end{equation}
where $M(t)$ is a smooth family of $2n\times 2n$ \emph{symmetric} matrices. If we couple the equation with a symmetric initial datum, then the solution must be symmetric as well on the maximal interval of definition. All the comparison results are based upon the following theorems. For a proof of these facts we refer to \cite[Appendix A]{BR-comparison}.
\begin{theoremappendix}[Riccati comparison 1]\label{t:riccati}
Let $M_1(t)$, $M_2(t)$ be two smooth families of $2n\times 2n$ symmetric matrices. Let $X_i(t)$ be smooth solutions of the Riccati equation
\begin{equation}
\dot{X}_i = \mathrm{R}_i(X_i;t), \qquad i =1,2,
\end{equation}
on a common interval $I \subseteq \R$. Let $t_0 \in I$ and assume \emph{(i)} $M_1(t) \geq M_2(t)$ for all $t \in I$, \emph{(ii)} $X_{1}(t_0) \geq X_{2}(t_0)$. Then for any $t \in [t_0,+\infty) \cap I$, we have $X_1(t) \geq X_2(t)$.
\end{theoremappendix}

The assumptions of Theorem~\ref{t:riccati}  involve comparison on coefficients of Riccati equations and on initial data. It can be generalised for limit initial data as follows.
\begin{theoremappendix}[Riccati comparison 2]\label{t:riccatilim}
Let $M_1(t)$, $M_2(t)$ be two smooth families of $2n\times 2n$ symmetric matrices. Let $X_i(t)$ be smooth solutions of the Riccati equation
\begin{equation}
\dot{X}_i = \mathrm{R}_i(X_i;t), \qquad i =1,2,
\end{equation}
on a common interval $I \subseteq \R$. Let $t_0 \in \overline{I}$. Assume that \emph{(i)} $M_1(t) \geq M_2(t)$ for all $t \in \overline{I}$, \emph{(ii)} $X_i(t) >0$ for $t>t_0$ sufficiently small, \emph{(iii)} there exist $Y_i(t_0):=\lim_{t\to t_0+} X_i^{-1}(t)$ and \emph{(iv)} $Y_1(t_0) \leq Y_2(t_0)$. Then, for any $t \in (t_0,+\infty) \cap I$, we have $X_1(t) \geq X_2(t)$.
\end{theoremappendix}

The above results can be used to prove that the Cauchy problem with limit initial condition is well posed, in the following sense.
\begin{lemmaappendix}\label{l:limit}
Let $A,B$ be constant $n\times n$ matrices, with $B\geq 0$ and satisfying the Kalman condition
\begin{equation}
\rank(B,AB,\ldots,A^{m-1}B) = n,
\end{equation}
for some $m\geq 0$. Let $R(t)$ be a smooth family of symmetric $n\times n$ matrices. Then the Cauchy problem with limit initial condition
\begin{equation}\label{eq:riccatiappendix}
\dot{V} +A^{*}V + VA +VB V + R(t)=\mathbbold{0} , \qquad \displaystyle \lim_{t\to 0^+} V(t)^{-1} = \mathbbold{0},
\end{equation}
is well-posed and admits a unique solution defined on a maximal interval $I \subseteq (0,+\infty)$. This solution is symmetric, and $V(t) >0$ for small $t>0$. Furthermore, if $R(t) =R$ is constant, then $V(t)$ is non-decreasing.
\end{lemmaappendix}

\bibliographystyle{alphaabbrv}
\bibliography{SR-Bakry-Emery-biblio}

\end{document}